\documentclass[11pt,leqno,a4paper]{amsart}
\usepackage[OT2,T1]{fontenc}
\usepackage{amsfonts, amssymb, amsmath, amsthm, latexsym, enumerate, epsfig, color, mathrsfs, fancyhdr, pdfsync, eurosym, endnotes, stmaryrd, bm, bbm}
\usepackage[all]{xy}

\usepackage{graphicx}
\usepackage{epstopdf}
\usepackage{setspace}
\usepackage{pdfsync}
\usepackage{amsmath,amscd,amssymb}
\usepackage[english]{babel}
\usepackage[inner=2cm,outer=2cm,bottom=5cm]{geometry}
\usepackage{setspace}
\usepackage{mathrsfs}
\usepackage[all]{xy}
\usepackage{tikz-cd}

\newtheorem{thm}{Theorem}[section]
\newtheorem*{thm*}{Theorem}
\newtheorem{cor}[thm]{Corollary}
\newtheorem{lemma}[thm]{Lemma}

\newtheorem{defn}[thm]{Definition}

\theoremstyle{remark}

\theoremstyle{definition}

\newtheorem{rmk}[thm]{Remark}

\numberwithin{equation}{thm}

\def\beq{\begin{equation}}
\def\eeq{\end{equation}}

\def\ben{\begin{enumerate}}
\def\een{\end{enumerate}}

\usepackage[OT2,T1]{fontenc}
\DeclareSymbolFont{cyrletters}{OT2}{wncyr10}{m}{n}
\SetSymbolFont{cyrletters}{bold}{OT2}{cmr}{bx}{n}
\DeclareMathSymbol{\Dc}{\mathalpha}{cyrletters}{68}

\DeclareMathOperator*{\bigast}{\raisebox{-0.6ex}{\scalebox{2.5}{$\ast$}}}

\def\crash#1{}
\def\N{{\mathbb N}}

\def\P{{\mathbb P}}
\def\Q{{\mathbb Q}}
\def\R{{\mathbb R}}

\def\A{{\mathbb A}}

\def\l{\left}
\def\r{\right}
\def\[[{\l[\l[}
\def\]]{\r]\r]}

\def\cf{\emph{cf.}\;}
\def\ie{\emph{i.e.}\;}
\def\lc{\emph{loc.cit.}\;}

\def\cA{{\mathcal A}}

\def\cE{{\mathcal E}}
\def\cF{{\mathcal F}}

\def\cM{{\mathcal M}}

\def\cO{{\mathcal O}}

\def\cR{{\mathcal R}}

\def\cT{{\mathcal T}}

\def\cX{{\mathcal X}}
\def\cY{{\mathcal Y}}

\def\sC{{\mathscr C}}

\def\sH{{\mathscr H}}

\def\N{{\mathbb N}}

\def\0{{\mathbb O}}

\def\P{{\mathbb P}}
\def\Q{{\mathbb Q}}
\def\R{{\mathbb R}}

\def\A{{\mathbb A}}

\def\fX{{\mathfrak X}}
\def\fY{{\mathfrak Y}}
\def\fZ{{\mathfrak Z}}

\def\frs{{\mathfrak s}}
\def\fri{{\mathfrak i}}

\def\wtilde{\widetilde}
\def\what{\widehat}

\def\b{{\mathbbm{b}}}

\def\Spec{{\rm Spec\,}}

\def\sep{{\rm sep}}
\def\ins{{\rm ins}}

\def\ol{\overline}
\def\iso{\xrightarrow{\ \sim\ }}

\def\kc{{k^\circ}}
\def\kcc{{k^{\circ\circ}}}
\def\kt{\widetilde{k}}

\def\red{{\rm red}}

\def\vphi{\varphi}
\def\eps{\epsilon}

\def\ena{ (\cE,\nabla)}
\def\fna{(\cF,\nabla)}

\def\pf{{\mathbbm{p}}}

\begin{document}
\title[Uniformization of morphisms and $p$-dic differential equations]{Metric uniformization of morphisms of Berkovich curves via $p$-adic differential equations}
\author{Francesco Baldassarri and Velibor Bojkovi\'c}


\begin{abstract}
We consider a finite \'etale morphism $f:Y \to X$ of quasi-smooth Berkovich curves over a complete nonarchimedean  non-trivially valued field $k$, assumed algebraically closed  and of characteristic 0, and a skeleton $\Gamma_f=(\Gamma_Y,\Gamma_X)$
 of the morphism  $f$.    We prove that $\Gamma_f$ radializes $f$ if and only if $\Gamma_X$  controls   the pushforward of the constant $p$-adic differential equation $f_\ast(\cO_Y,d_Y)$. 

Furthermore, when $f$ is a finite \'etale morphism of open unit discs, we prove that $f$ is radial if and only if the number of preimages of a point $x\in X$, counted without multiplicity, only  depends on the radius of the point $x$. 
\end{abstract}
\maketitle
\tableofcontents

\section{Introduction}
One of the most important results concerning the structure of smooth projective $k$-algebraic curves, where $k$ is a complete, nonarchimedean, and nontrivially valued algebraically closed field of characteristic 0, is the semistable reduction theorem~: such curves admit a semistable model. In Berkovich approach to nonarchimedean geometry, this theorem has many avatars, and extends to a more general class of curves, namely to quasi-smooth $k$-analytic curves (close analogs of  classical Riemann surfaces in complex analytic geometry) via the notion of triangulation or, alternatively, of skeleton. Namely, if $X$ is a quasi-smooth $k$-analytic curve then, then it admits a skeleton \cite[Chapter 5]{ducros}. Here, a skeleton of a curve is a locally finite ``graph'' $\Gamma$ in $X$ such that $X\setminus\Gamma$ is a disjoint union of open (unit) discs (this pretty much resembles the classical situation where if $\cX$ is a Riemann surface and $\cT$ its triangulation, then $\cX\setminus\cT$ is a disjoint union of unit  open $k$-discs).

If we consider a finite morphism $f:\cY\to \cX$ of smooth projective $k$-algebraic curves, then a result of Coleman \cite{Col03} says that there exist semistable models of $Y$ and $X$, respectively, to which $f$ extends as a finite morphism. With no surprise, this result extends to finite morphisms $f:Y\to X$ of quasi-smooth $k$-analytic curves where it can be stated as follows~: there exists a skeleton of the morphism $\Gamma_f=(\Gamma_Y,\Gamma_X)$ where $\Gamma_Y$ and $\Gamma_X$ are skeleta of $Y$ and $X$, respectively, such that $\Gamma_Y=f^{-1}(\Gamma_X)$ (see Section \ref{sec: radialization}). Among the many consequences of this result, one in particular simplifies the study of the morphism $f$. Namely, for any open disc $D$ in $Y$ which is attached to $\Gamma_Y$ (meaning that the closure of $D$ in $Y$ intersects $\Gamma_Y$ in only one point), the restriction $f_{|D}$ is a finite morphism of open discs, and the image $f(D)$ is an open disc attached to $\Gamma_X$. Furthermore, for every open disc $D$ attached to $\Gamma_X$, $f^{-1}(D)$ is a finite disjoint union of open discs attached to $\Gamma_Y$. 

One may ask  to what extent does the skeleton $\Gamma_f$ of a morphism $f$ capture its  properties. Conversely, can one find a skeleton $\Gamma_f$ which ``controls'' the behavior  of $f$ on discs attached to $\Gamma_Y$, in such a way that  at least some   properties of $f$ over such a disc $D$ only depend on its boundary point of $D$ on $\Gamma_Y$~? We will show that this is the case for  all metric properties of the morphism $f$.
%

To be more precise, we introduce some terminology. Given a finite morphism $f:D_1\to D_2$ of open unit discs, we say that $f$ is \emph{radial} if for \emph{any} pair of \emph{compatible} coordinates $T$ on $D_1$ and $S$  on $D_2$, (\ie such that $f$ sends $T=0$ to $S=0$), the valuation polygon  of the expansion $S(T)$ of $f$ stays the same (Definition \ref{defn: radial discs}). In other words, for any $x \in D_1$, the radius of the point $f(x)$ only depends upon the radius of $x$. 
We call such a valuation polygon (or rather its multiplicative version) the {\em profile of $f$}. One of the main results of \cite{TeMU} is the existence of the {\em radializing skeleta}, \ie for a finite morphism $f:Y\to X$ of quasi-smooth $k$-analytic curves, there exists a skeleton $\Gamma_f=(\Gamma_Y,\Gamma_X)$ such that for any two discs $D_1$ and $D_2$ in $Y$ attached to the same point on $\Gamma_Y$, the restrictions $f_{|D_1}$ and $f_{|D_2}$ are radial morphisms and their profiles coincide.  \par
\smallskip 
The other half of our story concerns $p$-adic differential equations on quasi-smooth Berkovich $k$-analytic curves. The theory flourished in the past decade or so, in an effort of globalizing over  a  curve  convergence properties of solutions and index theorems of the operators, discovered by Dwork and Robba for equations on standard  affinoid and dagger affinoid domains in the projective line. The global approach to index theorems was then developed by Christol and Mebkhout in a series of important papers.  We refer to \cite{Ke2} for a systematic exposition, and a deep refinement, of the results known until the year 2010, or so.
A new interpretation of the Dwork-Robba radius of convergence and a related conjecture, due to  the senior author \cite{Cont},  then opened the way to a clean global understanding  of   convergence properties  of local solutions of differential equations 
 on a Berkovich curve \cite{Pu1, PoPu}. 
\par \smallskip
Let us shortly recall the results which are most important for the present  paper.  Let $X$ be a quasi-smooth Berkovich $k$-analytic curve, $(\cE,\nabla)$ be a coherent $\cO_X$-module of rank $r$ equipped with a connection $\nabla:\cE\to\cE\otimes \Omega^1_X$ 
(simply called a \emph{$p$-adic differential equation} from now on), and let $\Gamma=\Gamma_X$ be a skeleton of $X$. Then, to every $k$-rational (\ie of type 1) point $x\in X(k)$ one may associate an $r$-tuple $\cM\cR_{\Gamma}(x,\ena)=(\cR_1,\dots,\cR_r)$ of numbers in $(0,1]$, called the \emph{multiradius of convergence} of solutions of $\ena$ at $x$, in the following way. We pick the unique open disc $D =: D_\Gamma(x,1^-)$ which contains $x$ and  is attached to $\Gamma$~: we call  $D$   \emph{the open $\Gamma$-unit disc centered at $x$}, so that the graph $\Gamma$ plays the role of a global unit of measurement. Then, the number $\cR_i$, for $i=1,\dots,r$ is the supremum of numbers $s \in (0,1)$ such that there are at least $r-i+1$ solutions of $\ena$ on the open subdisc $D_\Gamma(x,s^-)$ of $D$ centered at $x$ and of relative radius $s$ (see Section \ref{sec: push} for more details). The definition of multiradius extends to all   points of the curve $X$, with the method of  \cite[\S 0.1]{Cont}. The fundamental result is that the multiradius is a continuous function on $X$~: this was proven in \cite{Cont} for the the component $\cR_1$ and in \cite{Pu1, PoPu} in general. Furthermore, it is proved in \cite{PoPu} that there exists a skeleton $\Gamma'$ containing $\Gamma$ such that 
for any open disc $D$ in $X\setminus \Gamma'$, attached to the point $\xi \in \Gamma'$, $\cM\cR_{\Gamma}(x,\ena)=\cM\cR_{\Gamma}(\xi,\ena)$, for any $x \in D$. 
We say that such a $\Gamma'$ is \emph{a controlling skeleton for $\ena$} with respect to $\Gamma$. 
Notice that if  $\Gamma'$ is  a controlling skeleton for $\ena$  with respect to $\Gamma$, it is so with respect to $\Gamma'$, as well \cite[\S 3.2]{Cont}. A particular case is when $X$ is an open unit disc, so one can take $\Gamma = \emptyset$ as a skeleton of $X$. Then  $\emptyset$ is controlling for $\ena$ with respect to $\emptyset$ precisely when  the multiradius function is constant all over $X$.\\
\par \smallskip
The aim of the present article is to study the relation between radializing skeleta of a finite \'etale morphism  $f:Y\to X$ of quasi-smooth $k$-analytic curves and  controlling graphs of the $p$-adic differential equation $f_\ast(\cO_Y,d_Y)$ on $X$. Our main result is the following (Theorem \ref{thm: main thm}).
\begin{thm*}
 Let  $f:Y\to X$ be a finite  \'etale morphism of quasi-smooth $k$-analytic curves and let $\Gamma_f=(\Gamma_Y,\Gamma_X)$ be a skeleton for $f$. Then $\Gamma_f$ is radializing for $f$ if and only if $\Gamma_X$ is controlling for $ f_\ast(\cO_Y,d_Y)$ with respect to $\Gamma_X$, where $(\cO_Y,d_Y)$ is the constant $p$-adic differential equation on $Y$.
\end{thm*}
The close relation between radial morphisms and pushforwards of the constant connection has already been studied in \cite{BoPoPush}. There, the multiradius $\cM\cR_\Gamma(x,f_\ast(\cO_Y,d_Y))$ at a rational point $x\in X(k)$ has been described in terms of the profile of the restriction of $f$ on the connected components of $f^{-1}(D_{\Gamma_X}(x,1^-))$
(all of them  
open $\Gamma_Y$-unit discs).
Our result above further clarifies this relation.
\par \smallskip
Our first main ingredient is  Lemma~\ref{lem: push constant} 
below which indicates how the multiradius of convergence of solutions of $f_\ast(\cO_Y,d_Y)$ at $x \in D_2(k)$ is related to the jumps of the function ``cardinality of the fiber $f^{-1}(x_\rho)$''.
Secondly, we need a criterion of radiality for $f$ expressed in terms of a function on the \emph{target disc}.   
We end up with the following simple characterization  of radial morphisms of open unit discs 
(\cf Theorem~\ref{thm: criterion target} below)
\begin{thm*}
 A finite \'etale morphism of open unit discs $f:D_1\to D_2$ is radial if and only if, for any point $x \in D_2$, the 
 cardinality of the fiber $f^{-1}(x)$ only depends on the radius of $x$. 
 \end{thm*}
Notice that our statement is harder to prove than Lemma~2.3.6 of \cite{TeMU} which relates instead 
   radiality of  $f$ to radiality of the function ``multiplicity of $f$'' on the \emph{source disc}.     
\par
\medskip
We now describe the contents of the paper. In section 2 we recall some properties of finite morphisms of open discs. In particular, we introduce the notion of \emph{(weakly) $n$-radial} morphism which generalizes the one of radial morphism. From a careful study of those,  we  obtain the criterion of radiality for morphisms of open  discs  presented in Section 2.3. In order to prove our main result we further need to simplify the situation at a point  $\eta \in Y$ of type 2, and to reduce  to the case when $f$  is a   morphism  of affinoid curves with good reduction and maximal points $\eta$ and $f(\eta)$, respectively, which is residually purely inseparable at $\eta$, as this is the case when our criterion for radiality applies. So, we discuss the problem of when a   finite morphism  factors into 
a product of a residually purely inseparable morphism followed by a residually separable one  (see Section 2.5 for the result and definitions involved). In Section 3 we recall the general properties of $p$-adic differential equations and in Section 4 we prove our main result.

\section{Some properties of morphisms of open discs}

\subsection{Morphisms of open discs}
\subsubsection{}  
Throughout the paper $(k,|\cdot |)$ will be an algebraically closed complete valued field extension of $(\Q_p,|\cdot |_p)$. 

By an open (resp. closed) disc (or $k$-disc for precision) of radius $r \in \R_{> 0}$ we mean a $k$-analytic curve (in the sense of Berkovich geometry) $D$ isomorphic to a standard open (resp. closed) disc centered at 0 and of some radius $r>0$ in the affine $T$-line $\A^1_k =\cM(k[T])$. For any $k$-analytic domain $D' \subset D$ which is also a disc, the \emph{relative radius} of $D'$ in $D$ is well-defined. Similarly, to any point $x \in D$ we intrinsically associate the relative radius $r(\xi) = r_D(\xi)$ of $\xi$. Then, for any $k$-rational point  $a \in D(k)$ and $s \in (0,1)$, $D(a,s^-)$ (resp. $D(a,s)$)
will denote a disc  of relative radius $s$ in $D$ containing $a$ and   $\zeta_{a,s} \in D$, or simply $a_s$, will indicate the maximal point of $D(a,s)$. Similarly, for $0<r_1\leq r_2 \leq r$, we will denote by $A(a;r_1,r_2)$ (resp.  $A[a;r_1,r_2]$ if $r_2 <r$) an open (resp. closed) annulus centered at $a$ and with inner radius $r_1$ and outer radius $r_2$. That is $A(a;r_1,r_2)=D(a,r_2^-)\setminus D(a,r_1)$. 
Most often we will deal with a \emph{unit} disc, namely a disc $D$ 
equipped with a fixed  isomorphism $T:D\iso D(0,1^-)$ (resp. $D(0,1)$) in which case the previous notions coincide with 
the ones defined in terms of the coordinate $T$. 
In an open unit disc $D$, for any $a \in D(k)$, there is a unique \emph{path from $a$ to the exit}, namely  $l_a:=\{a_r \mid r \in [0,1)\}$ which we equip  in a natural way with the topology of a real segment. 
\subsubsection{} Let $D$ be an open (resp. closed) unit disc with  coordinate $T$  and let $f(T)=\sum_{i\geq 0}a_i\, T^i$ be an analytic function on $D$. We recall that the function $v(f,\cdot):(0,\infty)\to \R$, ( resp. $[0,\infty)\to\ \R$) defined by 
$$
\lambda \mapsto \inf_{i\geq 0}\{v(a_i)+i\cdot \lambda\}=-\log\big(\sup_{\substack{a\in k\\ |a| \leq e^{-\lambda}}}\{|f(a)|\}\big) = -\log (r(f(\zeta_{0,e^{-\lambda}})) )
$$
 is called the \emph{valuation polygon} of the function $f$. We recall some of the basic properties of the valuation polygon functions that will be used throughout this paper, while for a more detailed study we refer the reader to \cite{lazard}. 
 \begin{enumerate}
  \item $v(f,\cdot)$ is a continuous, piecewise affine and concave function. Its slopes are integral and in fact nonnegative. For $\lambda \in \R_{\geq 0}$ we denote by $\partial^+v(f,\lambda)$ (resp. $\partial^-v(f,\lambda)$) the right (resp. the left) slope of $v(f,\cdot)$ at $\lambda$.
  \item The values $\lambda \in \R_{\geq 0}$ such that $\partial^+v(f,\lambda)\neq \partial^-v(f,\lambda)$ are necessarily elements of $v(k^\times)$, called the \emph{break values} of the valuation polygon of $f$. The number $\partial^-v(f,\lambda)-\partial^+v(f,\lambda)$ is the number of zeroes of $f(T)$, counted with multiplicities, of valuation $\lambda$ (\ie of  absolute value $e^{-\lambda}$). 
  \item The valuation polygon is invariant under automorphisms of the disc $D$,  $T\to h(T)$, which preserve the origin (\ie such that $h(0)=0$).
 \end{enumerate}

\subsubsection{} Let $f:D_1\to D_2$ be a quasi-finite morphism of open unit discs and let $T$ and $S$ be coordinates on $D_1$ and $D_2$, respectively. Then $f$ can be expressed in the form 
\begin{equation}\label{eq: TS expression}
 S= S(T) = \sum_{i\geq 0}a_i\, T^i,\quad a_i\in k,
\end{equation}
where the coefficients $a_i$ satisfy the usual convergence property
$$
\lim\limits_{i\to \infty}|a_i|r^i=0 \;\;,\;\; \forall r\in (0,1) \;.
$$

We call \eqref{eq: TS expression} the $(T,S)$ expansion of $f$ and denote the power series on the right-hand side of it by $f_{(T,S)}(T)$. If $a_0=0$, we will say that $(T,S)$ is a compatible pair of coordinates for $f$. We denote the valuation polygon of the right-hand side of \eqref{eq: TS expression} by $v_{(T,S)}(f,\cdot)$ and call it the $(T,S)$-valuation polygon of $f$. 

Let $a \in D_1(k)$ be a $k$-rational point of $D_1$, let $b = f(a) \in D_2(k)$, and  $T$ be a coordinate on $D_1$ such that $T(a)=0$. We will say that $T$ is centered at $a$. Let $S$ be a coordinate on $D_2$ centered at $f(a)$ (note that $T$ and $S$ are then compatible). We note that the $(T,S)$-valuation polygon of $f$ then only depends on the point $a$ and not on the coordinates $T$ and $S$ (the property (3) above). We call it the \emph{valuation polygon of $f$ at $a \in D_1(k)$}, and denote it by $\lambda \mapsto v_a(f, \lambda)$, $\forall \, \lambda \in \R_{\geq 0}$.

\subsubsection{}  
We next explain the geometric meaning of the terms $\partial^+v(f,\lambda)$ and $\partial^-v(f,\lambda)$ introduced above. Let $f:D_1\to D_2$ be a quasifinite morphism of open unit discs, let $a\in D_1(k)$, $\lambda\in (0,\infty)$ so that $r:=e^{-\lambda} \in (0,1) \cap |k|$, $y = \zeta_{a,r}$, and $x = f(y) = \zeta_{f(a),r'}$.  The finite extension $\wtilde{\sH(y)}/\wtilde{\sH(x)}$ of function fields over $\kt$ determines a finite morphism $\wtilde{f}$ of smooth projective $\kt$-curves $\wtilde{f} : \sC_y \to \sC_x$.  Closed points of  $\sC_y$ (resp. $\sC_x$) correspond to open discs in $D_1$ (resp. $D_2$), attached to $y$ (resp. $x$). Following a suggestive picture, we regard such points or discs as ``tangent vectors''   on $Y$ at $y$ (resp. on $X$ at $x$); in particular, we denote by $\vec t_{y,a} \in \sC_y (\kt)$ the point corresponding to 
the open disc $D(a,r^-)$. 
By the theory of Newton polygons, the number $\partial^+ v_a(f,\lambda)$  is the degree of the finite morphism 
$D(a,r^-) \to D(f(a),{r'}^-)$ induced by $f$, \ie   the number of zeros counting multiplicities of $f$ in $D(a,r^-)$. On the other hand (\cf \cite[Th\'eor\`eme 4.3.13]{ducros}) the latter number  coincides  with 
the algebraic multiplicity of   $\wtilde{f}$ at   $\vec t_{y,a}$. 
In our case, both the curves $\sC_y$  and $\sC_x$ are projective lines over $\kt$, equipped with the affine coordinate  
$\wtilde{T}_a := \frac{T-a}{\pi} \mod \kcc$, $\wtilde{S}_{f(a)} := \frac{S-f(a)}{\pi'} \mod \kcc$.  Then $\wtilde{f}$ is represented in the coordinates  $\wtilde{T}_a$ and $\wtilde{S}_{f(a)}$ as a polynomial of degree 
$\partial^- v_a(f,\lambda)$. In particular, 
$$   [\wtilde{\sH(y)}:\wtilde{\sH(x)}] = [\sH(y):\sH(x)] 
$$ 
(the equality following from the fact that the valued field $\sH(x)$ is stable)  is the sum of the multiplicities at all the tangent points on $Y$ at $y$ and
coincides with  the geometric ramification index $\nu_f(y)$ of the point $y$ in the sense of \cite[\S 6.3.]{BerkovichEtale}. 
The same argument used at $\vec t_{y,a}$, \ie at $\wtilde{T}_a = 0$, shows that the
multiplicity of $\wtilde{f}$ at $\wtilde{T}_a = \infty$, \ie at the tangent vector $\vec t_{y,\infty}$  represented by the annulus $D_1 - D(a,r)$,  is the negative of $\partial^- v_a(f,\lambda)$.
Summing up the multiplicities of the zeros of $\wtilde{T}_a$ on $\sC_y =\P^1_{\kt}$ we obtain the classical proof of harmonicity of
the function $x \mapsto -\log |f(x)|$  on $D_1$. 
 \par
The previous discussion   proves   assertion {\em (1)} in the following lemma, while {\em (2)} and {\em (3)} are   not difficult to prove using the properties 1) and 2) of valuation polygons.

\begin{lemma}\label{lem: basic discs}\hfill  \ben
\item Let $f:D_1\to D_2$ be a quasi-finite morphism of open unit discs, let $a\in D_1$ and $r\in (0,1)$. Then $f_{|D(a,r^-)}:D(a,r^-)\to f(D(a,r^-))$ (resp. $f_{|D(a,r)}:D(a,r)\to f(D(a,r))$) is a finite morphism of open (resp. closed) discs of degree $\partial^+v_a(f,r)$ (resp. $\partial^-v_a(f,r)$).  
\item If $f$ is finite, then $f^{-1}(D(a,r^-))$ (resp. $f^{-1}(D(a,r))$) is a finite disjoint union of open (resp. closed) discs in $D_1$ and restriction of $f$ to each of them is a finite morphism to $D(a,r^-)$ (resp. $D(a,r)$).
\item The morphism $f$ induces a continuous increasing  bijection between the sets $l_a$ and $l_{f(a)}$, given by $r\mapsto r'$, where $r'$ is such that $D(f(a),r')=f(D(a,r))$.
\een
 \end{lemma}
 
 In the light of the Lemma we give some definitions.
 \begin{defn}
  Let $f:D_1\to D_2$ be a finite morphism of open unit discs and $a\in D_1(k)$. Let $T$ and $S$ be compatible coordinates for $f$ where $T$ is centered at $a$. Then, we call the function 
  $$
  \pf_{a,f}=\pf_{(T,S),f}:[0,1]\to [0,1]\quad \text{ given by } \pf_{a,f}(\rho)=\begin{cases}0 \quad \text{if}\quad \rho=0,\\
1 \quad \text{if}\quad \rho=1,\\
\rho' \quad \text{otherwise},
\end{cases}
 $$ 
 where $\rho'$ is such that $f(D(a,\rho)=D(a,\rho'))$, the $(T,S)$-profile of $f$ or the profile of $f$ at $a$.   
 \end{defn}

 \begin{rmk}\label{rmk: profile valuation}
 The relation between the profile of $f$ at $a$ and the valuation polygon of the morphism $f$ at $a$ is given by 
$$
\forall r\in (0,1),\quad v_{a}(f,-\log r)=-\log \pf_{a,f}(r).
$$ 
From this relation one concludes, having in mind the basic properties of valuation polygons, that $\pf_{a,f}$ is a continuous, piecewise monomial and increasing function. 
\end{rmk}
 \begin{lemma}\label{lem: mult relation}  
  Let $f:D_1\to D_2$ be a finite morphism of open unit discs and let $x\in D_2$ be a point of type 2. Then
  $$
  \sum_{y \in f^{-1}(x)}\nu_f(y) = \deg(f).
  $$
 \end{lemma}
 \begin{proof}
 See \cite[Remark 6.3.1.]{BerkovichEtale}.
 \end{proof}

\begin{cor}\label{cor: num preimages}
 Keeping the notation as in the lemma, if all the preimages of $x$ have the same geometric ramification index, say $\nu$, then $\#f^{-1}(x)=\deg(f)/\nu$.
\end{cor}

%

\subsection{(Weakly) $n$-radial morphisms}
 \begin{defn}\label{defn: radial discs}
 Let $f:D_1\to D_2$ be a finite morphism of open unit discs. We say that it is \emph{radial} if the functions $v_{a}(f,\cdot)$ (or equivalently, the functions $\pf_{a,f}$) are the same for all $a\in D_1$. If $f$ is radial  we will simply write $v(f,\cdot)$ and $\pf_f$ instead of $v_{a}(f,\cdot)$ and $\pf_{a,f}$, respectively, and call the latter function the \emph{profile} of $f$.
\end{defn}

\begin{rmk}\label{rmk: radiality and multiplicity}
If $f:D_1\to D_2$ is radial, and $\rho\in (0,1)$, then for any $a\in D_1(k)$ $\nu_f(\zeta_{a,\rho})$ does not depend on $a$. Indeed, by definition $\nu_f(\zeta_{a,\rho})=\partial^-v_a(f,-\log \rho)=\partial^-v(f,-\log \rho)$ which does not depend on $a$. 
\end{rmk}

\begin{rmk}
 Radial morphisms of open discs were first introduced in \cite[Section 2.3.]{TeMU} to which we refer for their main properties. 
\end{rmk}

It follows from the definition that to check whether the morphism $f$ is radial one picks,   for any point $a \in D_1(k)$,  a pair of compatible coordinates $(T,S)$ for $f$, where $T$ is centered at $a$ and then  compares the profile functions $\pf_{a,f}$.  
An obvious choice of compatible coordinates, for any $a\in D_1(k)$, 
is
$T_a:=T-T(a)$ and $S_{f(a)}:=S-S(f(a))$. 
The $(T_a,S_{f(a)})$-expansion of $f$ is given by 
$$
S_{f(a)}=\sum_{i\geq 1} f^{[i]}(a)T^i_a, \quad\text{where}\quad f^{[i]}(T):=\frac{1}{i!}\frac{\partial^i f(T)}{\partial T^i}.
$$
In this way, a finite morphism $f:D_1\to D_2$ of open unit discs is radial if and only if the valuation polygon of the function  $\sum_{i\geq 1} f^{[i]}(a)T^i_a$ is the same for all $a\in D_1(k)$.

\subsubsection{} A generalization of radial morphisms are the (weakly) $n$-radial ones (\cf \cite{BojFact}).
\begin{defn} We say that a finite morphism $f:D_1\to D_2$ of open unit discs is \emph{$n$-radial}, where $n\in \N$ ($0\notin \N$), if there exists a number $r\in (0,1)$ such that:
\begin{enumerate}
 \item for every $a\in D_1(k)$ the restriction of $v_{a}(f,\cdot)$ on $(0,-\log r)$ does not depend on $a$;
 \item $v_{a}(f,\cdot)$ has exactly $n$ slopes on $(0,-\log r)$.
\end{enumerate}
The infimum of the numbers $r$ above is denoted by $\b_{f,n}$ and is called the \emph{border of $n$-radiality}. The slopes of the valuation polygon of $f$ at $a$ (independent of $a\in D_1(k)$) over $(0,-\log r)$ are called the $n$-\emph{dominating terms} of $v_a(f,\cdot)$  and we denote their set by $\Dc_{f,n}$.  
Finally, by a \emph{$0$-radial} morphism we will simply mean a finite one.
\end{defn}

\begin{defn}\label{def: weakly n}
 Let $f:D_1\to D_2$ be a finite morphism of open discs. We will say that $f$ is \emph{weakly $(n+1)$-radial}, where $n\in \N$, if the following holds:
 \begin{enumerate}
  \item $f$ is $n$-radial.
  \item The first $n+1$ slopes of the valuation polygon $v_{a}(f,\cdot)$ do not depend on the choice $a\in D_1(k)$. The set of these slopes will still be  denoted by $\Dc_{f,n+1}$. 
 \end{enumerate}

\end{defn}
As we see every $n$-radial morphism is weakly $n$-radial (for $n>0$), while weakly $n$-radial does not necessarily
 imply   $n$-radiality, as is shown in Remark \ref{rmk: only weakly}.

\begin{lemma}\label{lem: weakly 1}
 Every finite morphism of open unit discs $f:D_1\to D_2$ is weakly $1$-radial.
\end{lemma}
\begin{proof}
 The morphism is trivially $0$-radial by definition. Suppose that $f$ is of degree $d$, and let $(T,S)$ be any pair of compatible coordinates on $D_1$ and $D_2$. The highest (which is the first) slope of $v_{(T,S)}(f,\cdot)$ is then necessarily equal to $d$, as this represents the number of solutions of the equation $f_{(T,S)}=c$, for any $c\in D_2(k)$. The claim follows.
\end{proof}

\begin{rmk}\label{rmk: only weakly} Assume that the field $k$ is of mixed characteristics $(0,p)$, where $p>2$. Let $f:D_1\to D_2$ be a finite morphism of open unit discs given by $S=T^{2p}+\alpha T$, where $1>|\alpha|>|p|$. Then $f$ is \'etale and for $a\in D_1(k)$, its $(T_a,S_{f(a)})$ expansion is given by 
$$
S_{f(a)}=T^{2p}_a+\sum_{i=1}^{2p-1}\binom{2p}{i}a^{2p-i}\,T_a^i+\alpha T_a.
$$
 If $a$ is such that $|a|^p>|\alpha|$, then since $\left|\binom{2p}{i}\right|=1$, for $i=p$ and $i=2p$, and $\left|\binom{2p}{i}\right|=|p|$ otherwise, it follows that the slopes of the valuation polygon of $f$ at $a$ are $1,\, p$ and $2p$ and the two break points $b_1$ and $b_2$ are given by $b_1=-\log|a|$ and $b_2=-\frac{1}{p-1}(\log|\alpha|-p\log|a|)$. Obviously, they both vary with $|a|$ so in particular, $f$ is not $1$-radial.
\end{rmk}

%

\subsubsection{}

We point out that,  if $f:D_1\to D_2$ is an $n$-radial morphism of open unit discs, then for any $a\in D_1(k)$, $| f^{[i_n]}(a)|$ does not depend on $a$, so we will write $|f^{[i_n]}|$ instead of $|f^{[i_n]}(a)|$ in what follows.

\begin{defn}
 Let $f:D_1\to D_2$ be a finite \'etale morphism of open unit discs, and suppose that it is weakly $n$-radial. Let $i_1>\dots>i_n$ be all the elements in $\Dc_{f,n}$. We define the function $\theta_n:D_1(k)\to [0,1)$, given by 
 $$
 a\in D_1(k)\mapsto \min\bigg\{\rho\in [0,1)\mid \max_i\Big\{|f^{[i]}(a)|\rho^i\Big\}=|f^{[i_n]}|\rho^{i_n}\bigg\},
 $$
 and call it \emph{the exact $n$-boundary function}. 
 \end{defn}
In other words, either $\theta_n(a)=0$, either $-\log(\theta_n(a))$ is the $n$-th break of the valuation polygon $v_a(f,\cdot)$ (see Figure \ref{fig: valuation polygon}). 
 
\begin{figure}[ht] 
\centering
\begin{picture}(330,220)(-50,0)     
\put(0,10){\vector(1,0){300}}       
\put(20,0){\vector(0,1){200}}         
\thicklines              
\put(0,180){\makebox(0,0)[b]{$v_a(f,\lambda)$}}
\put(20,10){\line(1,2){41}}

\put(60,90){\makebox(0,0){$\ast$}}
\put(60,90){\line(1,1){41}}

\put(100,130){\makebox(0,0){$\ast$}}
\put(140,145){\makebox(0,0){$\ast$}}
\put(140,145){\line(3,1){50}}
\put(200,173){\makebox(0,0)[b]{}}
\put(190,162){\makebox(0,0){$\ast$}}
\put(190,162){\line(5,1){50}}
\put(180,-5){\makebox(0,0)[b]{$-\log(\theta_n(a))$}}
\put(15,45){\makebox(0,0)[b]{slope $i_1$}}
\put(55,105){\makebox(0,0)[b]{slope $i_2$}}
\put(143,153){\makebox(0,0)[b]{slope $i_n$}}
\put(180,167){\makebox(0,0)[b]{slope $i_{n+1}(a)$}}

\put(265,175){\makebox(0,0){$\dots$}}

\put(122,135){\makebox(0,0){$\dots$}}
\put(180,0){\makebox(0,0)[b]{}}

\put(290,0){\makebox(0,0)[b]{$\lambda$}}
\put(190,10){\dashbox{2}(0,151){}}
\put(141,10){\dashbox{2}(0,135){}}
\put(60,10){\dashbox{2}(0,80){}}
\put(100,10){\dashbox{2}(0,120){}}

\end{picture}
\caption{The valuation polygon $v_a(f,\cdot)$.}\label{fig: valuation polygon}
\end{figure}
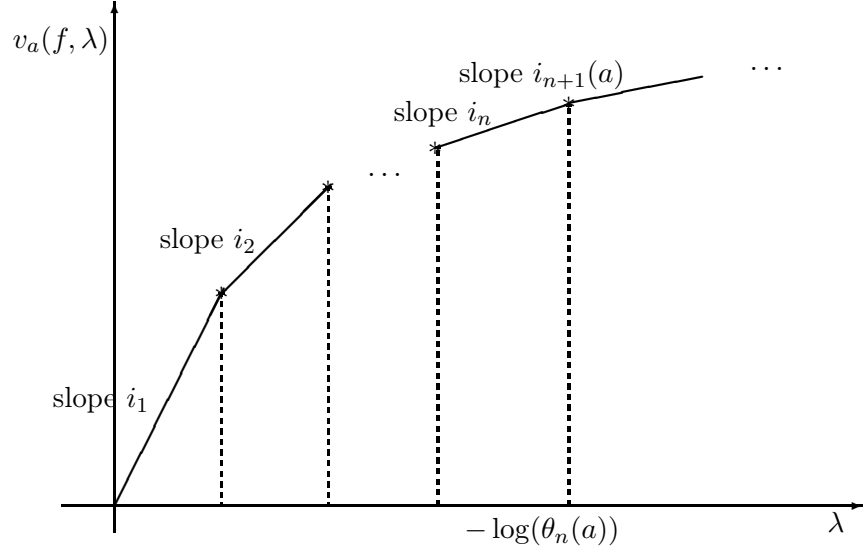
 
\subsubsection{}\label{rmk: important}
 We list some properties of the function $\theta_n$.
 
  (1) $\theta_n(a)$ does not depend on the chosen compatible coordinates on $D_1$ and $D_2$ with respect to which we calculate the terms $f^{[i]}(a)$ in the definition. This follows from property (3) of valuation polygons.
  
  (2) If the exact border of $n$-radiality is 0, then the morphism is radial. 
 
  (3) For $a\in D_1(k)$, the first $n$ slopes of the $(T_a,S_{f(a)})$-valuation polygon of $f$ are in $\Dc_{f,n}$. Then, there are two possible cases: 1) $\theta_n(a)=0$ which is equivalent to say $i_n=1$ (since $f$ is \'etale, the smallest slope is always 1) and the morphism $f$ is radial; 2) $\theta_n(a)\neq 0$, hence $i_n>1$. In this latter case we define the function $i_{n+1}:D_1(k)\to \N$ where $i_{n+1}(a)$ is the first next slope of the $(T_a,S_{f(a)})$-valuation polygon of $f$ (see Figure \ref{fig: valuation polygon}). In other words, $i_{n+1}(a)$ is minimal index $i_{n+1}(a)$ with the following property:
  $$
  |f^{[i_{n+1}(a)]}(a)|(\theta_n(a))^{i_{n+1}(a)}=|f^{[i_n]}|(\theta_n(a))^{i_n}. 
  $$
  Consequently we have
  $$
  \theta_n(a)=\left(\frac{|f^{[i_{n+1}(a)]}(a)|}{|f^{[i_n]}|}\right)^{\frac{1}{i_n-i_{n+1}(a)}}.
  $$
  
  (4) Suppose now and until Corollary \ref{cor: theta const} that $i_n>1$ so that $\theta_n(a)>0$ for every $a\in D_1(k)$. In this case we have, for every $a\in D_1(k)$, $\theta_n(a)=\rho$, where $\rho$ is such that 
  \begin{equation}\label{eq: special case}
  \max_{1\leq i<i_n}\{|f^{[i]}(a)|\rho^i\}=|f^{[i_n]}|\rho^{i_n}.
  \end{equation}

  For each $i=1,\dots,i_n-1$ let us define the functions 
  \begin{align}
  \theta_{n,i}:D_1(k)&\to [0,1)\nonumber\\
  a&\mapsto \rho \quad \text{such that}\quad  |f^{[i]}(a)|\rho^i=|f^{[i_n]}|\rho^{i_n}\nonumber,
  \end{align}
  that is
  \begin{equation}\label{theta n i}
  \theta_{n,i}(a)=\left(\frac{|f^{[i]}(a)|}{|f^{[i_n]}|}\right)^{\frac{1}{i_n-i}}.
  \end{equation}
   Formula \eqref{eq: special case} then implies that, for every $a\in D_1(k)$
  \begin{align*}
   \theta_n(a)&=\max_{1\leq i<i_n}\{\theta_{n,i}(a)\}\nonumber\\
   &=\max_{1\leq i<i_n}\Big\{\left(\frac{|f^{[i]}(a)|}{|f^{[i_n]}|}\right)^{\frac{1}{i_n-i}}\Big\}.\label{eq: theta n new}
  \end{align*}

  (5) The coefficient functions $f^{[i]}(T)$ are analytic functions on the open unit disc and from the valuation polygon theory it follows that  the following property holds: for every interval (assumed always to be connected) $I'\subset [0,1)$, there exists a subinterval $I=(r_1,r_2)\subset I'$ such that for every $a\in A(0;r_1,r_2)(k)$ and every $i=1,\dots,i_n-1$, we have $|f^{[i]}(a)|=|f^{[i]}(T)|_{|a|}$.
  
  Moreover, by shrinking $I$ if necessary, we may assume that for every $i\neq j$ and both in $\{1,\dots,i_n-1\}$ we have
  $$
  \left(\frac{|f^{[i]}(T)|_{|a|}}{|f^{[i_n]}|}\right)^{\frac{1}{i_n-i}}\neq \left(\frac{|f^{[j]}(T)|_{|a|}}{|f^{[i_n]}|}\right)^{\frac{1}{i_n-j}}.
  $$
  
 In this way we have proven 
  \begin{lemma}\label{lem: special intervals}
   Let $I'\subset (0,1)$ be an interval. Then, there exists a subinterval $I=(r_1,r_2)\subset I'$ and an $i\in \{1,\dots,i_n-1\}$ such that for every $a\in A(0;r_1,r_2)(k)$, $i_{n+1}(a)=i$ and 
   $$
   \theta_n(a)=\left(\frac{|f^{[i]}(T)|_{|a|}}{|f^{[i_n]}|}\right)^{\frac{1}{i_n-i}}.
   $$
  \end{lemma}

  Keeping notation as above, we also have
  
  \begin{lemma}\label{lem: theta constant}
   Suppose there exists an interval $(r_1,r_2)\subset (0,1)$ such that $\theta_n$ is constant on $A(0;r_1,r_2)(k)$. Then, both functions $i_{n+1}$ and $\theta_n$ are constant on $D(0,r_2^-)$. 
  \end{lemma}
\begin{proof}
 We may assume that $i_n>1$. 
 
 
 By the previous Lemma, there exists an interval $(r_1',r_2')\subset (r_1,r_2)$ and $j_0\in \{1,\dots,i_n-1\}$ such that for every $a\in A(0;r_1',r_2')(k)$, $i_{n+1}(a)=j_0$ and
 $$
   \theta_n(a)=\left(\frac{|f^{[j_0]}(T)|_{|a|}}{|f^{[i_n]}|}\right)^{\frac{1}{i_n-j_0}}.
 $$
 Since $\theta_n(a)$ is constant on $A(0;r_1',r_2')(k)$, then so is $|f^{[j_0]}(T)|_{|a|}$. Then, by the property (1) of valuation polygons, $|f^{[j_0]}(T)|_{|a|}$ is constant for all $a\in D(0,r_2'^{-})$ and then so is $\theta_{n,j_0}(a)$. Suppose that for some $a_1\in D(0;r_2'^-)(k)$ we have that $i_{n+1}(a_1)=j_1\neq j_0$, so that $\theta_{n,j_1}(a_1)>\theta_{n,j_0}(a_1)$ or $\theta_{n,j_1}(a_1)=\theta_{n,j_0}(a_1)$ and $j_1<j_0$. Again by the property (1) of valuation polygons there exists an $a_2\in A(0;r_1',r_2')(k)$ such that $|f^{[j_1]}(a_2)|\geq |f^{[j_1]}(a_1)|$. Then, by \eqref{theta n i} 
 \begin{align*}
 \theta_{n,j_1}(a_2)&\geq \theta_{n,j_1}(a_1)>\theta_{n,j_0}(a_1)=\theta_{n,j_0}(a_2)=\theta_n(a_2),\quad \text{or}\\
 \theta_{n,j_1}(a_2)&\geq \theta_{n,j_1}(a_1)=\theta_{n,j_0}(a_1)=\theta_{n,j_0}(a_2)=\theta_n(a_2) \quad \text{ and $j_1<j_0$}
 \end{align*}
 which is a contradiction in both cases. Hence, $i_{n+1}(a)=j_0$ and $\theta_n$ is constant on all of $D(0,r_2'^-)$. Finally we note that we could have chosen the interval $(r_1',r_2')$, so that $r_2'$ is arbitrarily close to $r_2$, again by Lemma \ref{lem: special intervals}, which ends the proof. 
%
\end{proof}

%
%
%
%
%

\begin{cor}\label{cor: theta const}
 Let $f:D_1\to D_2$ be a weakly $n$-radial morphism of open unit discs.
\begin{enumerate}
\item If $\theta_n(a)=0$, for every $a\in D_1(k)$, then $f$ is radial. 
\item Otherwise, $f$ is weakly $(n+1)$-radial if and only if there exists an $\eps>0$ such that the restriction of $\theta_n$ on $A(0;1-\eps,1)(k)$ is constant.
\end{enumerate}
 \end{cor}
 \begin{proof}
  The first point is clear. For the second point, if $f$ is weakly $(n+1)$-radial, then it is $n$-radial hence $\theta_n$ is constant on the whole disc $D_1(k)$. In the other direction, from Lemma \ref{lem: theta constant} it follows that $\theta_n(a)$ is constant for every $a\in D_1(k)$ which means that $f$ is $n$-radial, with border of radiality equal to $\theta_n(a)$ (for any $a\in D_1(k)$). The same lemma implies the constancy of the function $i_{n+1}(a)$, which is the second condition in Definition \ref{def: weakly n}.
  \end{proof}

 \subsection{A criterion for radiality}  
 
%

\begin{thm}\label{thm: criterion target}
 Let $f:D_1\to D_2$ be a finite \'etale morphism of open unit discs of degree $d$. Then, $f$ is radial if and only if the following holds: 
 There exists a function $N:[0,1)\to \N$ such that for every rational point $y\in D_2(k)$ we have $\#f^{-1}(y_\rho)=N(\rho)$. 
 
Moreover, the profile $\pf_f$ is uniquely determined by the function $N$.
\end{thm}
\begin{proof}
Let $(T,S)$ be a pair of compatible coordinates for $f$.

 Suppose that $f$ is a radial morphism of degree $d$, and let $\pf_f$ be its profile. Let $x_1,\dots,x_d$ be all the preimages of the point $y$. Then, all the preimages of the point $x_\rho$ are of the form $\zeta_{x_i,\rho_i}$, $i=1,\dots,d$ (some of which may coincide). Since $\pf_f(\rho_i)=\rho$, and $\pf_f$ is bijective the numbers $\rho_1,\dots,\rho_d$ are all equal and  $\pf_f^{-1}(\rho) = \{\rho_1=\dots=\rho_d\}$. 
 
 The multiplicity of each point $\eta_{x_i,\rho_i}$ is then equal to 
 $$
 \nu_f(\eta_{x_i,\rho_i})=\partial^-v_{x_i}(f,\pf_f^{-1}(\rho)),\quad i=1,\dots,d,
 $$
 and, since the right hand side only depends on $\rho$ and not on $x_i$ due to radiality of $f$, we also have $\nu_f(\zeta_{x_1,\rho_1})=\dots=\nu_f(\zeta_{x_d,\rho_d})$. Corollary \ref{cor: num preimages} then implies that $\#f^{-1}(y_\rho)=\frac{d}{\nu_f(\eta_{x_i,\rho_1})}$. Clearly, this number only depends on $\rho$ and not on $y$: this  
 is then our function $N(\rho)$. This proves the ``only if'' part of the statement. 
\par In the other direction, suppose we are given a function $N$ satisfying the conditions of the theorem. If $f$ is radial, we are done, so suppose that $f$ is not radial. By Lemma \ref{lem: weakly 1}, $f$ is weakly 1-radial. 
 \par Let $n$ be the maximal number such that $f$ is weakly $n$-radial, but not weakly $(n+1)$-radial. 
Then in particular $i_n>1$, where $i_n$ is the minimal element in $\Dc_{f,n}$. Let $(r_m,r_m')$ be a sequence of subintervals of $(0,1)$, satisfying the following properties (see Lemma \ref{lem: special intervals}):
 \begin{enumerate}
  \item $(r_{m+1},r_{m+1}')\subset (r_m',1)$;
  \item $\lim_{m\to \infty} r_m=1$ and
  \item there exists $j_m\in \{1,\dots,i_n-1\}$ such that for every $a\in A(0;r_m,r_m')(k)$
  $$
  \theta_n(a)=\left(\frac{|f^{[j_m]}(T)|_{|a|}}{|f^{[i_n]}|}\right)^{\frac{1}{i_n-j_m}}.
  $$ 
 \end{enumerate} The formula shows that, for sufficiently big $m$ and for $a \in A(0;r_m,r_m')(k)$,  $\theta_n(a)$  
 only depends on $|a|$ and increases with $|a|$. 
On the other hand, for sufficiently big $m$, $\theta_n$ cannot be constant on the annulus $A(0;r_m,r_m')(k)$ because otherwise it would be constant on all of the disc $D_1(k)$, due to Lemma \ref{lem: theta constant} and, by Corollary~\ref{cor: theta const}, $f$ would be $(n+1)$-radial, against the assumption. Hence, there exists $m_0$ such that $\theta_n$ is not constant on $A_{m_0}:=A(0;r_{m_0},r_{m_0}')(k)$. Moreover, we can choose $m_0$ arbitrarily large so that in particular for every $a\in A_{m_0}$, $|f(a)|=|a|^d$. This latter condition is equivalent to $f(A[0;|a|,|a|])=A[0;|a|^d,|a|^d]$ and $f^{-1}(A[0;|a|^d,|a|^d])=A[0;|a|,|a|]$. 

Let $r\in(0,1)$ be such that there exist $a,b\in A_{m_0}$, $|a|<|b|$ and 
\begin{equation}\label{eq: r}
\pf_{(T_a,S_{f(a)}),f}(\theta_n(a))<r<\pf_{(T_b,S_{f(b)}),f}(\theta_n(b)),
\end{equation}
or, in the other words, let us choose $r$ such that
$$
v_b(f,-\log(\theta_n(b)))<-\log r<v_a(f,-\log(\theta_n(a))),
$$
as is shown in the Figure \ref{fig: finding r}.
\begin{figure}[ht] 
\centering
\begin{picture}(330,220)(-50,0)     
\put(0,10){\vector(1,0){300}}       
\put(20,0){\vector(0,1){230}}         
\thicklines              
\put(20,10){\line(1,2){41}}

\put(60,90){\makebox(0,0){$\ast$}}
\put(60,90){\line(1,1){41}}

\put(100,130){\makebox(0,0){$\ast$}}
\put(140,145){\makebox(0,0){$\ast$}}
\put(140,145){\line(2,1){100}}
\put(240,195){\makebox(0,0){$\ast$}}
\put(200,173){\makebox(0,0)[b]{}}
\put(190,170){\makebox(0,0){$\ast$}}
\put(288,200){\makebox(0,0){$v_a(f,\lambda)$}}
\put(288,175){\makebox(0,0){$v_b(f,\lambda)$}}
\put(190,170){\line(6,1){40}}
\put(180,-5){\makebox(0,0)[b]{$-\log(\theta_n(b))$}}
\put(250,-5){\makebox(0,0)[b]{$-\log(\theta_n(a))$}}
\put(15,45){\makebox(0,0)[b]{slope $i_1$}}
\put(55,105){\makebox(0,0)[b]{slope $i_2$}}
\put(146,158){\makebox(0,0)[b]{slope $i_n$}}
\put(20,182){\makebox(0,0){$\ast$}}
\put(-5,182){\makebox(0,0){$-\log(r)$}}

\put(260,175){\makebox(0,0){$\dots$}}
\put(260,195){\makebox(0,0){$\dots$}}

\put(122,135){\makebox(0,0){$\dots$}}
\put(180,0){\makebox(0,0)[b]{}}

\put(294,-3){\makebox(0,0)[b]{$\lambda$}}
\put(190,10){\dashbox{2}(0,161){}}
\put(141,10){\dashbox{2}(0,135){}}
\put(60,10){\dashbox{2}(0,80){}}
\put(100,10){\dashbox{2}(0,120){}}
\put(240,10){\dashbox{2}(0,185){}}
\put(20,182){\dashbox{2}(185,0){}}
\put(20,195){\dashbox{2}(220,0){}}
\put(20,170){\dashbox{2}(170,0){}}

\end{picture}
\caption{The valuation polygons $v_a(f,\cdot)$ and $v_b(f,\cdot)$.}\label{fig: finding r}
\end{figure}
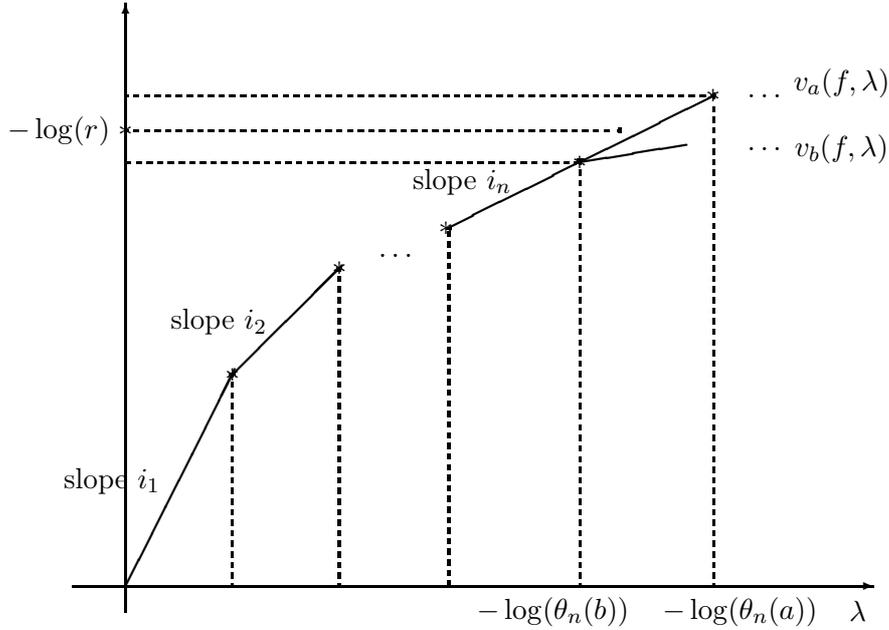
We note that we can always find such an $r$ because $\theta_n$ is not constant on $A_{m_0}$ and increases with the absolute value of the argument, and furthermore, $f$ being weakly $n$-radial and our choice of $A_{m_0}$ imply that the $(T_a,S_{f(a)})$ and $(T_b,S_{f(b)})$-profiles of $f$ coincide on the segment $[\theta_n(b),1]$ (that is, $v_a(f,\cdot)$ and $v_b(f,\cdot)$ coincide on the segment $(0,-\log(\theta_n(b))]$). Let $y_1:=\zeta_{f(a),r}$ and $y_2:=\zeta_{f(b),r}$. We note that $y_1$ and $y_2$ have the same radius, so they have the same number $N(r)$ of preimages, counted without multiplicities. 

We next study the preimages of the points $y_1$ and $y_2$. Let $a=a_1,\dots,a_d$ and $b=b_1,\dots,b_d$ be all the preimages of the points $f(a)$ and $f(b)$, respectively. Our choice of $A_{m_0}$ implies that for $i=1,\dots,d$, all the points $a_i$ have the same norm as $a$, while all the points $b_i$ have the same norm as $b$. Lemma \ref{lem: basic discs} (3)   implies that for each $i=1,\dots,d$, there is exactly one preimage of the point $y_1$ (resp. $y_2$) on the canonical path $l_{a_i}$ (resp. $l_{b_i}$), which we denote by $\zeta_{a_i,r_i}$ (resp. $\zeta_{b_i,s_i}$). Clearly, we have 
$$
\pf_{a_i,f}(r_i)=r\quad\text{ and }\quad \pf_{b_i,f}(s_i)=r. 
$$
Next, since $f$ is weakly $n$-radial and because of our choice of $A_{m_0}$, the values   $\theta_n(a_i)$, for $i=1,\dots,n$, all coincide.
In particular,  all the profile functions $\pf_{a_i,f}$, $i=1,\dots,n$, coincide on the segment $[\theta_n(a),1)$. Remark \ref{rmk: profile valuation} then implies, because of the first inequality in \eqref{eq: r}, that $r_1=\dots=r_n$. For the same reason, all the points $\zeta_{a_i,r_i}$ have the same geometric ramification index which is precisely $i_n$. Corollary \ref{cor: num preimages} then gives 
\begin{equation}\label{eq: ineq 1}
N(r)=\#f^{-1}(y_1)= \frac{d}{i_n}. 
\end{equation}
Similarly, the $(T_{b_i},S_{f(b_i)})$-profiles of $f$ coincide on the segment $[\theta_n(b),1]$. The second inequality in \eqref{eq: r} implies that $s_i<\theta_n(b)$, and consequently for each $i=1,\dots,n$, $\nu_f(\zeta_{b_i,s_i})\leq i_{n+1}(b_i)<i_n$. Let $\nu_2$ be the  maximal number among the $\nu_f(\zeta_{b_i,s_i})$, $i=1,\dots,n$. Lemma \ref{lem: mult relation} implies 
\begin{equation}\label{eq: ineq 2}
N(r)=\#f^{-1}(y_2)\geq \frac{d}{\nu_2}>\frac{d}{i_n}.
\end{equation}
Inequalities \eqref{eq: ineq 1} and \eqref{eq: ineq 2} give us a contradiction, hence $f$ is radial.
\par As for the last assertion of the theorem, we notice that to determine the profile $\pf_f$ amounts to determining the valuation polygon $v(f,\cdot)$ or, equivalently,  to  finding   breakpoints and corresponding slopes in between of the latter polygon (since we already know the behavior of the function $\lambda \mapsto v(f,\lambda)$ for $\lambda$ close to 0).
\par
If $0<b_1<\dots<b_n<1$ are the points of discontinuity of $N$ on the path $l_0$ from $0$ to the exit of $D$,    it is easy to see that $-\log b_n<\dots<-\log b_1$ are the breakpoints of $v(f,\cdot)$. Moreover, if $\rho\in (b_i,b_{i+1})$ for $i=1,..,n-1$ (resp. $\rho\in (0,b_1)$), then Remark \ref{rmk: radiality and multiplicity} and Corollary \ref{cor: num preimages} imply that the 
$\partial^-v(f,-\log \rho)=\frac{d}{N(\rho)}$.
 \end{proof}

\begin{defn}
 Let $f:D_1\to D_2$ be a finite morphism of open unit discs and let $a\in D_2(k)$. We define the function $N_a:=N_{f,a  }:[0,1)\to \N$ by 
 $$
 N_a(\rho):=\#f^{-1}(\zeta_{a,\rho}).
 $$
\end{defn}
\begin{rmk}\label{rmk:invdesc} Let $d$ be the degree of the finite morphism $f$.
Let $a_1,\dots,a_n$, with $n \leq d$, be the distinct inverse images of $a$. The set of connected components of 
$f^{-1}(D(a,\rho^-))$ consists of  discs $D(a_i,\rho_i^-)$, for $i=1,\dots,n$, not necessarily distinct. Then the inverse images of $\zeta_{a,\rho}$ are  
among the points  $\zeta_{a_i,\rho_i}$, for $i=1,\dots,n$. 
We deduce from this that $N_a(\rho)$ coincides with (using the previous notation)
\ben
\item 
the number $N_1$ of distinct points  $\zeta_{a_i,\rho_i}$, for $i=1,\dots,n$; 
\item
the maximum number $N_2$ of connected components of the inverse image of a connected  affinoid domain in $D_2$ with good reduction and with maximal point $\zeta_{a,\rho}$;
\item
the maximum number $N_3$ of connected components of the inverse image of a connected  affinoid domain in $D_2$ with good reduction and with maximal point $\zeta_{a,\rho}$, containing $a$. 
\een
The equalities $N_1 = N_2 = N_a(\rho)$ and $N_3 \leq N_2$ are clear. We only need to show that $N_2 \leq N_3$. 
\par
In fact, let $A$ be a connected  affinoid domain in $D_2$ with good reduction and with maximal point $\zeta_{a,\rho}$
such that $A_1,\dots,A_{N_2} \subset D_1$
are  the distinct connected components of the inverse image of $A$ in $D_1$. Then any $A_j$ has good reduction and as maximal point one of the points $\zeta_{a_i,\rho_i}$. Conversely,  any point $\zeta_{a_i,\rho_i}$ 
belongs to exactly one of  the affinoids $A_1,\dots,A_{N_2}$. We may then re-index the affinoids $A_j$ and the points 
$\zeta_{a_i,\rho_i}$ in such a way that, for $j=1,\dots,N_2$, $\zeta_{a_j,\rho_j}$ is the maximal points of $A_j$. 
For $j=1,\dots,N_2$, let $\ol{A}_j := A_j \cup D(a_j,\rho_j^-)$. Then $\ol{A} := A \cup D(a,\rho^-)$ is a connected  affinoid domain in $D_2$ with good reduction and with maximal point $\zeta_{a,\rho}$ which contains $a$, 
and $\ol{A}_j$  is a connected affinoid with good reduction with maximal point $\zeta_{a_j,\rho_j}$,  and it is a connected component of $f^{-1}(\ol{A})$. Moreover, $\ol{A}_1,\dots,\ol{A}_{N_2}$ are disjoint. We conclude  that $N_2 \leq N_3$. 
\end{rmk}
It follows from the previous remark that, for any $a\in D_2(k)$,  the function $N_a$ is non-increasing, right-continuous,  and has finitely many jumps (= points of discontinuity). Moreover, for $\rho$ close to 1, $N_{a}(\rho)=1$ while for $\rho$ close to 0, $N_a(\rho)\leq \deg(f)$, with strict inequality if and only if $a$ is a branching point. 

An immediate consequence of these properties and the previous theorem is the following
\begin{cor}\label{cor: N properties}
Let $f:D_1\to D_2$ be a finite morphism of open unit discs. 
\begin{enumerate}
 \item For each $a\in D_2(k)$, $N_a$ is uniquely determined by its jumps and by the values it takes at them.
 \item The morphism $f$ is radial if and only if for every two points $a,b\in D_2(k)$ the functions  $N_a$  and $N_b$ coincide. Moreover, the profile of $f$ is uniquely determined by the function $N_a$.
 \end{enumerate}
\end{cor}

We will return to the functions $N_a$ in Lemma \ref{lem: push constant}, where we will study its close relation with multiradius of pushforwards of the constant $p$-adic differential equations (see Section 2).

\subsection{Radializing skeleton of a morphism}\label{sec: radialization}

\subsubsection{}

For any quasi-smooth $k$-analytic curve $X$, the \emph{topological skeleton} $S(X)$ of $X$ is the complement of the union of all open discs in $X$. If $X=A(0;r_1,r_2)$ is an open annulus, the topological skeleton $S(A)$ is homeomorphic to the open segment $(r_1,r_2) \subset \R$. 
\par 
Let $X$ be a quasi-smooth strictly $k$-analytic curve. It follows from the existence of triangulations of quasi-smooth curves (\cite[Chapter 5.]{ducros}) (or semistable reduction) that there exists a locally finite set $\cT$ of type 2 points in $X$, such that $X\setminus \cT$ is a disjoint union of open analytic domains each of which is isomorphic to an open disc or an open annulus (for our purposes we also consider a punctured open disc to be an open annulus of inner radius 0). The union of the set $\cT$ with all the topological skeleta of annuli which are connected components of $X\setminus \cT$ is called \emph{the (semistable) skeleton of $X$ with respect to $\cT$} and we denote it by $\Gamma_{\cT}$. 
The complement $X \setminus \Gamma_{\cT}$ is then a disjoint union of open discs. We also note here that the only case in which $X$ admits an empty topological or semistable skeleton is when $X$ is an open unit disc. 
\par
Now if $X$ is as above and $\Gamma$ a nonempty semistable skeleton of $X$, we define the {\em retraction} function $r_{\Gamma}:X\to \Gamma$ in the following way. If $x\in \Gamma$ then we set $r_{\Gamma}(x)=x$. If $x \in X\setminus \Gamma$, then the connected component of $X \setminus \Gamma$ containing $x$ admits a unique boundary point 
$z \in \Gamma$. 
In this case we set $r_{\Gamma}(x)=z$. In the latter situation we will say that $D$ is {\em attached} to the point $z$.

The following is an easy result that will be used in the proof of the Lemma \ref{lem: controlling rat}.
\begin{lemma}\label{lem: finding rat}
 Let $D$ be an open disc, $\Gamma$ a nonempty skeleton of $D$ and $x\in D$ a point in $D\setminus \Gamma$. Then, there exists $y\in D(k)$ with $r_{\Gamma}(x)=r_{\Gamma}(y)$. 
\end{lemma}
\begin{proof}
 Let $D'$ be an open disc which is a connected component of $D\setminus\Gamma$ that contains $x$. Then since $r_{\Gamma}(D')=r_{\Gamma}(x)$, any rational point $y\in D'$ will do the job.
\end{proof}

\subsubsection{} Let $f:Y\to X$ be a finite morphism of quasi-smooth strictly $k$-analytic curves. By a skeleton of the morphism $f$ we mean a pair $(\Gamma_Y,\Gamma_X)$ such that $\Gamma_Y$ (resp. $\Gamma_X$) is a skeleton of $Y$ (resp. $X$) and such that $f^{-1}(\Gamma_X)=\Gamma_Y$. 

The following  result is  well known.
\begin{thm}
Any finite  morphism $f:Y \to X$ of quasi-smooth (strictly) $k$-analytic curves admits a (nonempty) skeleton.
\end{thm}
 Let $f:Y\to X$ be a finite morphism of quasi-smooth strictly $k$-analytic curves and let $\Gamma_f=(\Gamma_Y,\Gamma_X)$ be a nonempty skeleton of $f$. Then it is a direct consequence of the definition of  skeleton of a morphism that for any open disc $D$ which is a connected component of $Y\setminus \Gamma_Y$, the restriction $f_{|D}:D\to D'$ is a finite morphism of open discs, and $D'$ is a connected component of $X\setminus \Gamma_X$. We recall that any such a disc $D$ can be identified with an open unit disc.
 \begin{defn}
  Let $f:Y\to X$ and $\Gamma_f=(\Gamma_Y,\Gamma_X)$ be as above. We say that the morphism $f$ is radial with respect to $\Gamma_f$ if for any two open discs $D_1$ and $D_2$ that are attached to the same point in $\Gamma_Y$ (that is $r_{\Gamma_Y}(D_1)=r_{\Gamma_Y}(D_2)$), the restrictions $f_{|D_1}$ and $f_{|D_2}$ are radial morphisms having the same profile function. 
 \end{defn}
 \begin{rmk}\label{rmk: 2 out 3}
  We will use the following, easily establishing fact (\cite[Lemma 3.3.13.]{TeMU}). Suppose that $f:Y\to Z$ and $g:Z\to X$ are two finite \'etale morphisms of quasi-smooth strictly $k$-analytic curves, and suppose that $\Gamma_f=(\Gamma_Y,\Gamma_Z)$ and $\Gamma_{g}=(\Gamma_Z,\Gamma_X)$ are their respective skeleta, so that $\Gamma_{g\circ f}=(\Gamma_Y,\Gamma_X)$ is a skeleton for $g\circ f$. Then, if two out of the three skeleta $\Gamma_f$, $\Gamma_g$ and $\Gamma_{g\circ f}$ are radializing, then so is the third one. 
 \end{rmk}

 \begin{rmk}
  One of the main results of \cite{TeMU} is the existence of a radializing skeleton for finite morphisms of quasi-smooth strictly $k$-analytic curves (\lc Theorem 3.4.11.). We will reprove this result by establishing a close relation between the radializing skeleta of a morphism and the controlling graphs of the pushforward of the constant connection by the morphism, which will be the subject of Sections \ref{sec: push} and \ref{sec: rad and contr}.
 \end{rmk}

 \subsection{Factorization of morphisms}

 \subsubsection{} We recall briefly some properties of reduction of affinoid curves. For more detail we refer  to \cite[Section 6.3]{BGR}, \cite[Section 2.4]{Berkovich} or to the   book project \cite{ducros}. If $X$ is a quasi-smooth, strictly $k$-affinoid curve, its canonical reduction  (\cite[Section 2.4]{Berkovich}), denoted by $\wtilde{X}$, is a $\kt$-algebraic affine curve. If $\cA_X$ is the corresponding affinoid algebra, let $\cA_X^\circ$ denote the $k^\circ$-algebra $\{f\in \cA\mid \sup_{x\in X}|f(x)|\leq 1\}$ and let $\cA_X^{\circ \circ}:=\{f\in \cA_X^\circ\mid \sup_{x\in X}|f(x)|<1\}$. Then, the $\wtilde{k}$-algebra of regular functions $\cO_{\wtilde{X}}$ on $\wtilde{X}$ is  $\cA_X^\circ/\cA_X^{\circ \circ}$ and $\wtilde{X}=\Spec \cA_X^\circ/\cA_X^{\circ \circ}$. 
 
 \par
 Let us denote the reduction map by $\red:X\to \wtilde{X}$. If $\wtilde{X}$ is smooth, we say that $X$ has  (canonical)  good reduction. In this case, the Shilov boundary of $X$ consists of a single point (\cite[Proposition 2.4.4.]{Berkovich}). 
\par 
 Let $f:Y\to X$ be a finite morphism of quasi-smooth strictly $k$-affinoid curves with good reduction with maximal points $\eta$ and $\xi$, respectively. 
 We already recalled that $X\setminus\{\xi\}$ is a disjoint union of open unit discs, each of which is attached to the point $\xi$. In this case the map $\red$ induces a 1-1 correspondence between the smooth points of $\wtilde{X}$ and the connected components of  $X\setminus\{\xi\}$ (\cite[Theorem 4.3.1]{Berkovich}, \cite[Section 4.2.11.1]{ducros}). 
 For any  $y \in Y$, we have $\red(f(y))=\wtilde{f}(\red(y))$. 
  If $D$ is any disc in $Y$ attached to $\eta$, then $f(D)$ is a disc in $X$ attached to $\xi$ and for every disc $E$ in $X$ attached to $\xi$, $f^{-1}(E)$ is a disjoint union of discs in $Y$, attached to $\eta$. 
 We will use freely this correspondence in what follows. 
 \par



 \subsubsection{}\label{sec: factorization reduction} 
 Recall that if $\wtilde{f}:\wtilde{Y}\to \wtilde{X}$ is a finite morphism of smooth  connected $\wtilde{k}$-algebraic curves, then $\wtilde{f}$ factors canonically as
 \begin{equation}\label{eq: red fact}
 \wtilde{Y}\xrightarrow{\wtilde{f}_\ins} \wtilde{Z}\xrightarrow{\wtilde{f}_\sep}\wtilde{X}
 \end{equation}
 where $\wtilde{f}_\ins:\wtilde{Y}\to \wtilde{Z}$ is a finite, radicial morphism while $\wtilde{f}_\sep:\wtilde{Z}\to \wtilde{X}$ is finite and generically \'etale. The factorization in fact corresponds to the field extensions $\kappa(\wtilde{X})\subset \kappa(\wtilde{Z})\subset \kappa(\wtilde{Y})$, where  $\kappa(\wtilde{Z})$ is the separable closure of $\kappa(\wtilde{X})$ in $\kappa(\wtilde{Y})$. More precisely, we have $\wtilde{Z} \iso \wtilde{Y}^{(p^r)}$, where $\wtilde{Y}\to \wtilde{Y}^{(p^r)}$ is the $r$-fold relative Frobenius morphism and where $p^r$ is the degree of $\kappa(\wtilde{Y})$ over $\kappa(\wtilde{Z})$ (\cite[p. 291]{Liu} or \cite[Part 3, Prop. 50.13.7]{stacks}).   
  
 \begin{defn} \label{red-prop} 
  Let $f:Y\to X$ be a finite morphism of strictly  $k$-affinoid curves having good reduction. We say that $f$ is 
   a \emph{residually separable} (resp. \emph{residually radicial}, resp. \emph{residually \'etale})  morphism (at the maximal point of $Y$) if the reduced morphism $\wtilde{f}:\wtilde{Y}\to \wtilde{X}$ is generically \'etale (resp.  radicial, resp.  \'etale) morphism  of smooth affine $\wtilde{k}$-algebraic curves. 
   We put $\frs(f):=\deg(\wtilde{f}_\sep)$ and $\fri(f):=\deg(\wtilde{f}_\ins)$. 
  \end{defn}
    \begin{rmk} \label{red-int-pt1} Definition~\ref{red-prop} extends to the case of a quasi-finite morphism $f:Y\to X$ of quasi-smooth $k$-analytic curves at an internal  point of type 2,  $\eta \in Y$. In that case, reduction at $\eta$ is a morphism of smooth  projective $\kt$-curves and $f$ \emph{residually radicial} (resp. \emph{residually \'etale}) means that $\wtilde{f}$ is radicial (resp. \'etale) as such.
    \end{rmk}
  \begin{rmk}  \label{red-int-pt2} Definition~\ref{red-prop} and Remark~\ref{red-int-pt1} extend  to  a quasi-finite morphism $f:Y\to X$ of quasi-smooth $k$-analytic curves and to any point $\eta \in Y$ of type $>1$, by a suitable extension of scalars. See  \cite[Section 1.2.]{BoPoPush}. 
  This generalization is not needed for our present purposes.
  \end{rmk}

  The main result of this section is the existence of a lifting of the canonical factorization \eqref{eq: red fact} for a morphism of affinoid curves. We will be able to lift the factorization for the  class of morphisms described in the next definition and in the lemma that follows it.
  \begin{defn} \label{resunifram}
   Let $f:Y\to X$ be a finite morphism of quasi-smooth, strictly $k$-affinoid curves with good reduction. We say that $f$ is \emph{uniformly residually ramified} (at the maximal point of $Y$)  if the degree $\deg(f_{|D})$, where $D$ is any open disc in $Y$ attached to its maximal point, does not depend on $D$. This is  the case iff  the morphism $\wtilde{f}$ has the same multiplicity at every closed point $\wtilde{y}\in\wtilde{Y}$.
  \end{defn}
  
  \begin{lemma}\label{lem: uniform ram properties}
   Let $f:Y\to X$ be a finite \'etale morphism of strictly $k$-affinoid curves with good reduction. Then, $f$ is uniformly residually ramified if and only if   the following equivalent conditions hold: 
   \begin{enumerate}
   \item Let $\wtilde{f}=\wtilde{f}_\sep\circ\wtilde{f}_\ins$ be as in \eqref{eq: red fact}. Then, $\wtilde{f}_\sep$ is \'etale.
    \item For every open disc $D$ attached to the maximal point of $Y$, $\deg(f_{|D})=\fri(f)$ and for every open disc $E$ attached to the maximal point of $X$, the number of connected components of $f^{-1}(E)$ is equal to $\frs(f)$.
   \end{enumerate}
  \end{lemma}
  \begin{proof}
   {\em (1)} Let $\wtilde{y} \in \wtilde{Y}$ be a closed point and let $e_{\wtilde{f},\wtilde{y}}$ denote the algebraic multiplicity 
   of $\wtilde{f}$ at $\wtilde{y}$. Then from \eqref{eq: red fact}
   $$
   e_{\wtilde{f},\wtilde{y}}=\fri(f)\, e_{\wtilde{f}_\sep,\wtilde{f}_\ins(\wtilde{y})},
   $$
   so we see that uniformity of residual  ramification is equivalent to the fact that, for any closed point $\wtilde{y}\in\wtilde{Y}$, $e_{\wtilde{f}_\sep,\wtilde{f}_\ins(\wtilde{y})}$ is the same number,  necessarily $=1$. 
   This in turn means  that $\wtilde{f}_\sep$ is \'etale.
   
  {\em (2)} Continuing {\em (1)}, $\wtilde{f}_\sep$ \'etale is equivalent to the condition that every point $\wtilde{y}\in \wtilde{Y}(\wtilde{k})$ has the same multiplicity equal to $\fri(f)$. This is equivalent to that, for every point $\wtilde{x}\in \wtilde{X}(\wtilde{k})$,  $\#\wtilde{f}^{-1}(\wtilde{x})=\deg(\wtilde{f}_\sep)$ (using  $\sum_{\wtilde{y}\in \wtilde{f}^{-1}(\wtilde{x})}e_{\wtilde{f},\wtilde{y}}=\deg(\wtilde{f})$). Finally, this is equivalent to the condition that for every open disc $E$ attached to the Shilov point of $X$,
  $$
  \#f^{-1}(E)=\#\wtilde{f}^{-1}(\wtilde{x})=\frs(f).
  $$
  \end{proof}
  \begin{rmk} \label{rmk:counterexamples}  Let $f:Y \to X$ be a finite   morphism  of strictly $k$-affinoid curves with good reduction and let $\eta$ be the maximal point of $Y$.
  \ben
  \item 
   If $f$ is \'etale and residually separable at  $\eta$,  then it is residually \'etale and therefore also residually uniformly ramified at $\eta$. 
To prove this it will suffice  to show that if $f$ is residually separable at $\eta$,  $e_{\wtilde{f},\wtilde{y}} = 1$ for any $\wtilde{y} \in \wtilde{Y}(\kt)$. 
   We let $D$ be the open unit disc attached at $\eta$ and corresponding to $\wtilde{y}$. Let 
   $$T \in \cO_{Y,\eta} = \kappa(\eta) \subset \sH(\eta)
   $$ 
   be a coordinate on $D$. We express $f$ as a power series $f(T)$ with coefficients in $\kc$. Then $f$  \'etale at $\eta$ implies  that $\sup_{a \in D(k)}|(df/dT) (a)|=1$ while 
   $\wtilde{f}$ generically \'etale means that $d\wtilde{f}/d\wtilde{T} \neq 0$.  If $e_{\wtilde{f},\wtilde{y}} >1$ 
   we would have $(d\wtilde{f}/d\wtilde{T}) (\wtilde{y})  = 0$ and therefore the value of $|df/dT|$ could not be constant on $D$. By the theory of valuation polygons, $df/dT$ would have a zero in $D$. Hence, $f$ would not be \'etale in $D$. Contradiction. 
\item   If $f$ is \'etale, but not residually separable  at $\eta$, then $f$ need not be residually uniformly ramified at $\eta$.  For example, $S=a\, T+T^{2\,p}$, $a\in k^\circ$, $|p|<|a|$ is a finite \'etale morphism from the closed $T$-disc $D(0,1)$ to itself. However,  the separable part of $\wtilde{f}$ is ramified over $0$, hence the morphism is not residually uniformly ramified.
\item
  Notice that if $f$ is residually uniformly ramified  at an interior point of type 2 (necessarily of resdual genus $0$),  the reduced morphism is the product of a finite radicial morphism of a projective line over $\kt$ followed by a finite \'etale morphism of projective lines over $\kt$. But the latter is an isomorphism, so a map residually uniformly ramified  at an interior point of type 2 reduces to a power of  relative Frobenius. 
  \item We do not know whether in case $f$  is residually separable at $\eta$ but is ramified at some point $y \in Y(k)$, there may exist points of type 2 in the maximal open disc in $Y$ centered at $y$ at which $f$ is residually radicial.
      \een
  \end{rmk}

  Suppose now that $f:Y\to X$ is a finite morphism of strictly affinoid curves with good reduction which is radial with respect to the skeleton $(\{\eta\},\{\xi\})$ coming from the Shilov points of $Y$ and $X$, respectively. Then, by the definition of radiality, for every open disc $D$ in $Y$, attached to the Shilov point of $Y$, $\deg(f_{|D})$ is the same for all of them. Consequently,
  \begin{cor}\label{cor: radial implies uniformity}
   A radial morphism of strictly quasi-smooth $k$-affinoid curves with good reduction is residually uniformly ramified.
  \end{cor}
  
  In the other direction, we have
  \begin{cor}\label{cor: res etale radial}
 A residually \'etale morphism $f:Y\to X$ of strictly $k$-affinoid curves with good reduction is radial with respect to the skeleton coming from the Shilov points of $Y$ and $X$, respectively.
  \end{cor}
\begin{proof}
  Indeed, if $D$ is any disc in $Y$ attached to its Shilov point, the restriction $f_{|D}$ is an isomorphism (because it has degree $1$), hence is radial.
\end{proof}

\subsubsection{} We   may now factorize.
  
  \begin{thm}\label{thm: factorization}
   Let $f:Y\to X$ be a finite \'etale morphism of strictly $k$-affinoid curves with good reduction which is residually uniformly ramified. Then, there exists a strictly $k$-affinoid curve $Z$ with good reduction, together with finite \'etale morphisms $f_i:Y\to Z$ and $f_s:Z\to X$ such that $f=f_s\circ f_i$, $\wtilde{Z}\iso \wtilde{Y}^{(p^r)}$, $\wtilde{(f_s)}=\wtilde{f}_\sep$ and  $\wtilde{(f_i)}=\wtilde{f}_\ins$, where $\fri(f)=p^r$.
  \end{thm} 
\begin{proof}   The morphism $f:Y\to X$ canonically induces   a finite morphism
$\Phi: \fY \to \fX$  of affine smooth $\kc$-formal schemes topologically of finite type such that the reduction  
$\wtilde{f}: \wtilde{Y} \to \wtilde{X}$ identifies with the special fiber of $\Phi$. 
 In the reduction of $f$ we have the factorization 
 $$
 \wtilde{Y}\xrightarrow{\wtilde{f}_\ins} \wtilde{Z} \xrightarrow{\wtilde{f}_\sep} \wtilde{X} \;. 
 $$
 as in \eqref{eq: red fact}. By Lemma \ref{lem: uniform ram properties}, $\wtilde{f}_\sep$ is (finite and)  \'etale so  that,   by   \cite[Lemma 2.1]{BerkovichCycles} (see also \cite[Exp. I, Cor. 8.4]{sga1}) there exists an affine smooth $\kc$-formal scheme topologically of finite type $\fZ$ and  
 a finite \'etale morphism $\Phi_\sep : \fZ \to \fX$ with special fiber $\wtilde{f}_\sep$. 
  Let $Z$ be the generic fiber   of $\fZ$, so that  $Z$ is a quasi-smooth 
strictly  $k$-affinoid curve with good  reduction  $\wtilde{Z}$. 
The generic fiber of $\Phi_\sep$ is then a finite morphism of quasi-smooth strictly $k$-affinoid curves with good reduction $f_2:Z \to X$ whose reduction is $\wtilde{f}_\sep$.  
Now, we are in the domain of \cite[Theorem 1.1.]{Col85} (where one takes $W = \emptyset$) and we may conclude that there exists a lifting $f_1:Y\to Z$ of $\wtilde{f}_\ins$ such that we have $f=f_2\circ f_1$. It  follows by their construction that $f_1$ and $f_2$ satisfy the properties required  in the statement.
\end{proof}

\section{Pushforwards of the constant connection}\label{sec: push}
\subsection{Generalities on $p$-adic differential equations}

\subsubsection{} Let $X$ be a quasi-smooth strictly $k$-analytic curve. If $(\cE,\nabla)$ is a $p$-adic differential equation on $X$, by which we mean a locally free $\cO_X$-module $\cE$ of finite type and of rank $r$, equipped with an integrable connection $\nabla$, then for every skeleton $\Gamma$ of $X$, we can define the multiradius function, $\cM\cR_\Gamma:X(k)\to (0,1]^r$, in the following way.  

Let $x\in X(k)$ be a rational point. Then, there exists a unique maximal open disc, say $D_x$ in $X$, which is a connected component of $X\setminus \Gamma$ and such that $x\in D_x$. We can choose a coordinate $T$ on $D_x$ which identifies it with the standard open unit disc, and as such we have a well defined radius function on $D_x$. For $r\in (0,1)$ we denote as usual by $D(x,r^-)$ the open disc centered at $x$ and of radius $r$. This disc does not depend on the chosen coordinate $T$.

\begin{defn}
 Keeping the situation above, we define the \emph{multiradius of convergence} of solutions of $\ena$ at a rational point $x$, denoted by $\cM\cR_\Gamma(x,\ena)$, as the $r$-tuple of numbers 
 $$
 \cM\cR_\Gamma(x,\ena):=(\cR_{1,\Gamma}(x,\ena),\dots,\cR_{r,\Gamma}(x,\ena)),
 $$
 where $r$ is the rank of $\cE$ and $\cR_i:=\cR_{i,\Gamma}(x,\ena)$ is given by
 $$
 \cR_i:=\sup\{s\in (0,1)\mid \dim_k H^0(D(x,s^-),\ena)\geq r-i+1\}.
 $$
 Here $H^0(D(x,s^-),\ena)$ is the $k$-vector space of analytic functions on $D(x,s^-)$ that are in the kernel of $\nabla$.
 
 We extend the previous definition to any point $x\in X$ by extending the scalars to the completion $K$ of the algebraic closure of the completed residue field $\sH(x)$, extending the skeleton $\Gamma$ to a skeleton of $X\what{\otimes}K$, picking a suitable rational point in $X\what{\otimes}K$ which is ``above'' $x$ and repeating the previous procedure, as is done with more details in \cite[Definition 3.1.11]{Cont} or \cite[Section 2.2]{PoPu}.
\end{defn}
%

The number $\cR_1$ is commonly referred to as the radius of convergence of solutions of $\ena$ at the point $x$. 

\subsubsection{}
\begin{defn}
 We say that the skeleton $\Sigma\supset \Gamma$ of $X$ \emph{controls $\cM\cR_{\Gamma}(\cdot,\ena)$ (with respect to $\Gamma$)} if, in case $\Sigma\neq \emptyset$, for any point $x\in X$ we have
 $$
 \cM\cR_\Gamma(x,\ena)=\cM\cR_\Gamma(r_\Sigma(x),\ena).
 $$
 If $\Sigma=\emptyset$, this is taken to mean that $\cM\cR(\cdot, \ena)$ is a constant vector over $X$. 
\end{defn}
\begin{rmk}\label{rmk: restr disc}
If $X$ is a quasi-smooth strictly $k$-analytic curve, $\Gamma$ its skeleton, $\ena$ a $p$-adic differential equation on $X$ and $x\in X(k)$, then it follows from the definition of the multiradius of convergence that 
$$
\cM\cR_\Gamma(x,\ena)=\cM\cR_{\emptyset}(x,\ena_{|D}),
$$
where $D$ is the open unit disc in $X\setminus \Gamma$, attached to $\Gamma$ and that contains $x$. 
\end{rmk}

We introduce some notation. For two vectors $\vec{v}\in \R^n$ and $\vec{u}\in \R^m$, we denote by $\vec{v}\ast \vec{u}$ the vector $\vec{w}\in \R^{n+m}$, which is obtained from $\vec{v}$ and $\vec{u}$ by concatenation and arranging the coefficients in a nondecreasing order (for example, $(1,2,9)\ast (4,6)=(1,2,4,6,9)$). For a $n$-fold $\ast$-product of a vector $\vec{v}$ with itself we will write $\vec{v}^{\ast d}$.

A fundamental result on $p$-adic differential equations is that the multiradius of convergence is a continuous function on $X$ and that a controlling skeleton exists. We recall some of the properties that will be used in this article.
\begin{thm}\label{thm: properties p-de controlling}
 For any $X$, $\Gamma$ and $\ena$ as above, there exists a skeleton $\Sigma$ of $X$ that controls $\ena$ with respect to $\Gamma$. Moreover, the following holds:
 \begin{enumerate}
  \item The multiradius function is continuous as a function from $X\to (0,1]^r$, where $r$ is the rank of $\cE$. It is locally constant around type 1 and type 4 points. 
  \item Any skeleton $\Sigma'$ of $X$ that contains $\Sigma$, controls $\ena$ with respect to $\Gamma$.
  \item Suppose that $\ena=(\cE_1,\nabla_1)\oplus (\cE_2,\nabla_2)$. Then, for every $x\in X$,
  $$
  \cM\cR_{\Gamma}(x,\ena)=\cM\cR_\Gamma(x,(\cE_1,\nabla_1))\ast \cM\cR_\Gamma(x,(\cE_2,\nabla_2)).
  $$
  Moreover, if $\Sigma$ controls $\ena$, then it controls both $(\cE_1,\nabla_1)$ and $(\cE_2,\nabla_2)$.
  \end{enumerate}
\end{thm}
\begin{proof}
 The part one is in \cite[Theorem 0.1.7]{Cont} for $\cR_1$, and \cite[Theorem 3]{Pu1} and \cite[Theorem 3.6]{PoPu} in general. The local constancy around type 4 points is proved in \cite[Section 4.4]{Ke5}. The first part of {\em (3)} comes directly from the definition of the multiradius while the second part follows from the continuity of the multiradius function.
\end{proof}

For us, the following property of the controlling graphs will be particularly useful.

\begin{lemma}\label{lem: controlling rat}
Let $X$ be a quasi-smooth strictly $k$-analytic curve, $\Gamma$ a skeleton of $X$, let $\ena$ be a $p$-adic differential equation on $X$, and $\Sigma\supset\Gamma$ another skeleton of $X$. Then, $\Sigma\neq\emptyset$ is controls $\ena$ with respect to $\Gamma$ if and only if for every $x,y\in X(k)$ such that $r_\Sigma(x)=r_{\Sigma}(y)$, we have
$$
\cM\cR_\Gamma(x,\ena)=\cM\cR_\Gamma(y,\ena).
$$
\end{lemma}
\begin{proof}
 The "only if" part is clear. So, suppose that for every $x,y\in X(k)$ with $r_{\Sigma}(x)=r_{\Sigma}(y)$ the corresponding multiradii coincide. We note that a consequence of this condition and continuity of multiradius (Theorem \ref{thm: properties p-de controlling}) is that $\cM\cR_\Gamma(x,\ena)=\cM\cR_\Gamma(r_{\Sigma}(x),\ena)$. 
 
 Suppose, for the sake of contradiction, that $\Sigma$ is not controlling for $\ena$. This means that there exist a point $\xi\in X\setminus \Sigma$, such that 
 $$
 \cM\cR_\Gamma(\xi,\ena)\neq\cM\cR_\Gamma(r_{\Sigma}(\xi),\ena).
 $$
 By continuity of the multiradius, we may even assume that $\xi$ is of type 2. Let $D$ be a connected component of $X\setminus \Sigma$ which contains $\xi$ and let $\Sigma'$ be any controlling graph of $\ena$ which contains $\Sigma$. Necessarily, $\Sigma'_D:=D\cap\Sigma'\neq \emptyset$ and is a skeleton of $D$. By Lemma \ref{lem: finding rat} there exists a point $x\in D(k)$ such that $r_{\Sigma'_D}(x)=r_{\Sigma'_D}(\xi)$. Then, since $\Sigma'$ is controlling 
 \begin{align*}
 \cM\cR_{\Gamma}(\xi,\ena)&=\cM\cR_{\Gamma}(x,\ena)=\cM\cR_{\Gamma}(r_{\Sigma}(x),\ena)\\ &=\cM\cR_{\Gamma}(r_{\Sigma}(\xi),\ena)\neq  \cM\cR_{\Gamma}(\xi,\ena),
 \end{align*}
 which is a contradiction.
\end{proof}

\subsubsection{} For this section let $f:Y\to X$ be a finite \'etale morphism of degree $d$ of quasi-smooth strictly $k$-analytic curves. We recall that in this case $f_*\cO_Y$ is a locally free $\cO_X$-module of finite rank which is equal to the degree $\deg(f)$. More generally, if $\cE$ is a locally free $\cO_Y$-module of finite rank $r$, $\vphi_*\cE$ is a locally free $\cO_X$-module of finite rank $r\cdot d$. 

We also note that since $f$ is \'etale, we have an isomorphism $\Omega^1_Y\cong f^*\Omega^1_X$. Then if we are given $\ena$ a $p$-adic differential equation on $Y$ we may push forward by $f$ the integrable connection 
$$
\nabla:\cE\to \cE\otimes \Omega^1_Y,
$$
to obtain (using the projection formula) 
$$
f_*(\nabla):f_*\cE\to f_*(\cE\otimes \Omega^1_Y)\cong f_*(\cE\otimes f^*\Omega^1_X)\cong f_*\cE\otimes \Omega^1_X,
$$
hence an integrable connection on $f_*\cE$. We call the $p$-adic differential equation $(f_*\cE,f_*\nabla)$ on $X$ the \emph{pushforward of $\ena$ by $f$}. We refer to \cite[Section 1.]{Ram} for more details.

By a slight abuse of notation, we will write $\nabla$ for $f_*\nabla$ hoping it will be clear from the context which connection it denotes.

\begin{rmk}\label{rmk: push solutions} We note an important consequence of the definition of pushforward of $p$-adic differential equations. Namely (keeping the previous notation), if $U$ is any analytic subdomain of $X$, then we have an isomorphism of $k$-vector spaces
\begin{equation}\label{eq: push solutions}
 H^0(U,\fna)\iso H^0(f^{-1}(U),\ena).
\end{equation}
\end{rmk}

A natural question that arises is the relation between the multiradius of convergence of $\ena$ at some point $y\in Y$ and the one of $\fna$ at the point $x=f(y)$. The answer is given in \cite{BoPoPush} while in lemmas \ref{lem: push separable} and \ref{lem: push constant} we will recall two particular cases that we will use later on. 

\begin{lemma}\label{lem: push separable}
Let $f:Y\to X$ be a finite \'etale morphism of degree $d$ where $Y$ and $X$ are strictly $k$-affinoid curves with good reduction and with Shilov points $\eta$ and $\xi$, respectively. Let $\ena$ be a $p$-adic differential equation on $Y$ and $\fna$ be its pushforward on $Y$. Suppose that $f$ is residually \'etale, let $x\in X(k)$ and let $f^{-1}(x)=\{y_1,\dots,y_d\}$. Then
 \begin{equation}\label{eq: ast prod of radii}
 \cM\cR_{\{\xi\}}(x,\fna)=\bigast_{i=1}^d\cM\cR_{\{\eta\}}(y_i,\ena).
 \end{equation}
\end{lemma}
\begin{proof}
 Let $D$ be the connected component (open unit disc) of $X\setminus \{\xi\}$ that contains $x$. Since $f$ is residually \'etale, $f^{-1}(D)$ is a disjoint union of $d$ open discs, each of which is attached to $\eta$ and if $D'$ is any of them the restriction $f_{|D'}$ is an isomorphism of open unit discs (Lemma \ref{lem: uniform ram properties}). Then, for each $i=1,\dots,d$, there is a unique open disc, say $D_i$ which is a preimage of $D$, that contains $y_i$. We conclude that 
 $$
 \fna_{|D}=f_*(\ena_{|\bigcup_{i=1}^d D_i})=\bigoplus_{i=1}^d(f_{|D_i})_*(\ena_{|D_i}),
 $$
 so that by Theorem \ref{thm: properties p-de controlling} and Remark \ref{rmk: push solutions} 
 $$
\cM\cR(x,\fna_{|D})=\bigast_{i=1}^d\cM\cR(x,(f_{|D_i})_*(\ena_{|D_i}))=\bigast_{i=1}^d\cM\cR(y_i,\ena_{|D_i}).
 $$
 Finally, since by Remark \ref{rmk: restr disc}
 $$
 \cM\cR_{\{\xi\}}(x,\fna)=\cM\cR(x,\fna_{|D}) \quad \text{and} \quad \cM\cR_{\{\eta\}}(y_i,\ena)=\cM\cR(y_i,\ena_{|D_i}),
 $$
 we obtain \eqref{eq: ast prod of radii}.
\end{proof}

 \begin{cor}\label{cor: multi constant discs}
 Let $f:Y\to X$ be a finite \'etale, residually \'etale morphism of degree $d$ of strictly $k$-affinoid curves with good reduction. Let $\ena$ be a $p$-adic differential equation on $Y$ of rank $r$ and let $\fna$ be its pushforward on $X$. Let $\eta$ and $\xi$ be the Shilov points of $Y$ and $X$, respectively.
 
 Then, $\{\eta\}$ is controlling for $\ena$ with respect to $\{\eta\}$ if and only if $\{\xi\}$ is controlling for $\fna$ with respect to $\{\xi\}$.  
\end{cor}
\begin{proof}
  Let $x\in X(k)$ and suppose that $\{\eta\}$ is controlling for $\ena$. Let $y_1,\dots,y_d$ be all the preimages of $x$. Then, by Lemma \ref{lem: push separable}  
 $$
 \cM\cR_{\{\xi\}}(x,\fna)=\bigast_{i=1}^d\cM\cR_{\{\eta\}}(y_i,\ena)=\cM\cR_{\{\eta\}}(y_1,\ena)^{\ast d}=\cM\cR_{\{\eta\}}(\eta,\fna)^{\ast d},
 $$
 hence $\{\xi\}$ is controlling for $\fna$. 
 
 For the other direction, suppose that $\{\xi\}$ is controlling for $\fna$, and let $\Sigma$ be any controlling skeleton for $\ena$. If $\Sigma=\{\eta\}$  we are done so suppose that $\eta$ is properly contained in $\Sigma$. Since $\Sigma$ is a skeleton of $Y$, there are only finitely many connected components of $Y\setminus\{\eta\}$ that intersect $\Sigma$. Let us denote their union by $Z$. We note that for every $y\in Y\setminus Z$, $r_{\Sigma}(y)=\eta$ so that
 \begin{equation}\label{eq: first case}
\cM\cR_{\{\eta\}}(y,\fna)=\cM\cR_{\{\eta\}}(\eta,\fna),\quad \text{for all} \quad y\in Y\setminus Z.
 \end{equation}
Let $x\in X(k)\setminus f(Z)$. Then, keeping the notation $\{y_1,\dots,y_d\}=f^{-1}(x)$, the previous formula together with \eqref{eq: ast prod of radii} implies that 
\begin{equation}\label{eq: ast power d}
\cM\cR_{\{\xi\}}(x,\fna)=\cM\cR_{\{\eta\}}(\eta,\fna)^{\ast d}.
\end{equation}
Since $\{\xi\}$ is controlling for $\fna$, the previous formula is valid for all $x\in X(k)$.

  Now let $y\in Z(k)$ be arbitrary and let $x=f(y)$, and let $y_1=y,y_2,\dots,y_d$ be all the preimages of $x$. Formulas \eqref{eq: ast prod of radii} and \eqref{eq: ast power d} imply that all the components of the multiradius $\cM\cR_{\{\eta\}}(y,\ena)=\cM\cR_{\{\eta\}}(r_\Sigma(y),\ena)$ are contained in the discrete set of components of $\cM\cR_{\{\eta\}}(\eta,\ena)$. Let $D_y$ be the connected component of $Z$ that contains $y$. Since $\cM\cR_{\{\eta\}}(\cdot,\ena)$ is continuous along the skeleton $\big(\Sigma\cap D_y\big)\cup\{\eta\}$ (where the skeleton is equipped with the induced topology), and it takes values in the discrete subset of $(0,1]^{r}$, it must be constant, hence equal to $\cM\cR_{\{\eta\}}(\eta,\ena)$. Since $y$ was arbitrary rational point in $Z$, and having in mind \eqref{eq: first case} we conclude that for every $y_1,y_2\in Y(k)$, 
  $$
  \cM\cR_{\{\eta\}}(y_1,\ena)=\cM\cR_{\{\eta\}}(y_2,\ena),
  $$
  hence by Lemma \ref{lem: controlling rat} we conclude that $\{\eta\}$ itself is controlling for $\ena$.
\end{proof}

The following is Corollary 4.4 in \cite{BoPoPush}, but for the convenience we provide the full proof here.
\begin{lemma}\label{lem: push constant} (See \cite[Corollary 4.4.]{BoPoPush})
 Let $f:D_1\to D_2$ be a finite \'etale morphism of degree $d$ of open unit discs, let $\fna:=f_*(\cO_{D_1},d_{D_1})$ be the pushforward of the constant connection and let $x\in D_2(k)$. Let further $b_1<\dots<b_{n-1}$ be the jumps of the function $N_x$, let $b_0=0$ and put $N_i:=N_x(b_{i-1})$, $i=1,\dots,n$. Then, the multiradius
 $$
 \cM\cR(x,\fna)=(R_1,\dots,R_d),
 $$
 is given by
 \begin{align*}
  &R_1=\dots=R_{d-N_2}=b_1;\\
  &R_{d-N_2+1}=\dots=R_{d-N_3}=b_2;\\
  &\vdots\\
  &R_{d-N_{n-1}+1}=\dots=R_{d-1}=b_{n-1};\\
  &R_d=1.
 \end{align*}
\end{lemma}
\begin{proof}
We start by noticing that for $i=0,\dots,n-1$ and $s\in (b_i,b_{i+1})$,  $\#f^{-1}(D(x,s^-))=N_{i+1}$. Indeed, in this case $N_x(s)$ is constant and each connected component of $f^{-1}(D(x,s^-))$ is determined by the point in $D_1$ to which it is attached, and there is precisely $N_x(b_{i})=N_{i+1}$ of these. 

 From the definition of multiradius of convergence and using Remark \ref{rmk: push solutions} we obtain 
 \begin{align*}
 R_i&=\sup\{s\in (0,1)\mid \dim_k H^0(D(x,s^-),\fna)\geq d-i+1\}\\
 &=\sup\{s\in (0,1)\mid \dim_k H^0(f^{-1}(D(x,s^-)),(\cO_{D_1},d_{D_1}))\geq d-i+1\}.
 \end{align*}
  Next we note that 
\begin{align*}
\dim_k H^0(f^{-1}(D(x,s^-)),(\cO_{D_1},d_{D_1}))&=\sum_{D\text{ c.c. of }f^{-1}(D(x,s^-))}\dim_k H^0(D,(\cO_{D_1}, d_{D_1}))\\
&=\#f^{-1}(D(x,s^-)).
\end{align*} 
where "c.c." stands for "connected component" and we used that $H^0(D,(\cO_{D_1},d_{D_1}))=k$. Finally, we obtain 
\begin{equation}\label{eq: R_i simple}
R_i=\sup\{s\in (0,1)\mid \#f^{-1}(D(x,s^-))\geq d-i+1\}.
\end{equation}
Since by Remark \ref{rmk:invdesc} for $s$ close enough to 1, $\#f^{-1}(D(x,s^-))=1$, we immediately obtain that $R_d=1$. To find $R_{d-1}$ and the rest of the radii we use the remark from the beginning of the proof to see that the supremum of $s\in (0,1)$ such that $\#f^{-1}(D(x,s^-))\geq2$ is precisely $b_{n-1}$ and since $\#f^{-1}(D(x,b_{n-1}^-))=N_{n-1}$ we obtain that $R_{d-N_{n-1}+1}=\dots=R_{d-1}=b_{n-1}$. The same reasoning gives us the rest of the radii.
\end{proof}

The following lemma should be compared to the Corollary \ref{cor: radial implies uniformity} as it ``announces'' the relation between the controlling graphs of the pushforard of the constant connection and radializing skeleta of the morphism. 
\begin{lemma}\label{lem: push res ram}
 Let $f:Y\to X$ be a finite  \'etale morphism of strictly $k$-affinoid curves with good reduction and maximal points $\eta$ and $\xi$, respectively. Let $\fna:=f_*(\cO_Y,d_Y)$. 
 
 If $\{\xi\}$ is controlling for $\cM\cR_{\{\xi\}}(\cdot, \fna)$ then $f$ is residually uniformly ramified.  
\end{lemma}
\begin{proof}
 Let $x\in X(k)$, let $D_x$ be the maximal disc in $X$ attached to $\xi$ and which contains $x$. Let $D_1, \dots,D_s$ be all the preimages of the disc $D_x$. We will prove that $s$ does not depend on $D_x$.
 By equality 
 $$
\fna_{|D_x}=\bigoplus_{i=1}^s (f_{|D_i})_*(\cO_{D_i},d_{D_i}). 
$$
and Remark \ref{rmk: restr disc} it follows that 
$$
\cM\cR_{\{\xi\}}(x,\fna)=\cM\cR(x,\fna_{|D_x})=\bigast_{i=1}^s\cM\cR(x,(f_{|D_i})_*(\cO_{D_i},d_{D_i})).
$$
The last equality and Lemma \ref{lem: push constant} implies that there are exactly $s$ entries in $\cM\cR_{\{\xi\}}(x,\fna)$ which are equal to $1$. In particular, since $\{\xi\}$ is controlling for $\fna$, it follows that $s$ is the same for all $x\in X(k)$, hence the number of preimages of any open disc in $X$ that is attached to $\xi$ is constant. By Lemma \ref{lem: uniform ram properties} $f$ is residually uniformly ramified.
\end{proof}

\section{Radializing and controlling skeleta}\label{sec: rad and contr}

We are ready for our main result.

\begin{thm}\label{thm: main thm}
 Let $f:Y\to X$ be a finite \'etale morphism of quasi-smooth strictly $k$-analytic curves and let $\Gamma_f=(\Gamma_Y,\Gamma_X)$ be a skeleton of $f$. Then, $\Gamma_f$ is a radializing skeleton for $f$ if and only if $\Gamma_X$ is controlling for the connection $f_*(\cO_Y,d_Y)$ with respect to $\Gamma_X$.
\end{thm}

\begin{proof} 
Let us put $\fna:=f_*(\cO_Y,d_Y)$.

First we consider the case where $\Gamma_Y=\Gamma_X=\emptyset$ so that both $Y$ and $X$ are open unit discs. In this case, $f$ being radial is equivalent to for any $x,y\in X(k)$, $N_x\equiv N_y$, by Corollary \ref{cor: N properties} {\em (2)}. Then, by Corollary \ref{cor: N properties} {\em (1)} and Lemma \ref{lem: push constant} this is equivalent that $\cM\cR(x,\fna)=\cM\cR(y,\fna)$ so that $\cM\cR(\cdot,\fna)$ is constant all over $X$. 

Assume now that $\Gamma_Y$ and $\Gamma_X$ are not empty. 

Let $C_\xi$ be a strictly $k$-affinoid domain in $X$ with good reduction and with Shilov point $\xi\in \Gamma_X$ and such that $C_\xi\cap \Gamma_X=\{\xi\}$ (we note that the $k$-rational points of such $k$-affinoid domains cover $X(k)$). Since $\Gamma_f$ is a skeleton of $f$, $f^{-1}(C_\xi)$ is a disjoint union of affinoid domains $C_{\eta_i}$, where each $C_{\eta_i}$ is an affinoid domain in $Y$ with good reduction, with Shilov point $\eta_i\in \Gamma_Y$ and $C_{\eta_i}\cap \Gamma_Y=\{\eta_i\}$, for $i=1,\dots,n$. Then,
 \begin{equation}\label{eq: sum}
  \fna_{|C_\xi}=\bigoplus_{i=1}^nf_*(\cO_{C_{\eta_i}}, d_{C_{\eta_i}}).
 \end{equation}
Let us suppose that $f$ is radial and let $x\in C_{\xi}(k)$. Let $D$  be the connected component of $X\setminus \Gamma$  that contains $x$. So $D\subset C_\xi$ and  $f^{-1}(D)$ is a disjoint union of open discs in $Y\setminus \Gamma_Y$ each of which is attached to $\Gamma$. Les us denote, for $i=1,\dots,n$, by $D_{i,1},\dots,D_{i,l(i)}$  those of the previous discs which are contained in $C_{\eta_i}$, and let us write $f_{i,j}:=f_{|D_{i,j}}$. We note that by Corollary \ref{cor: radial implies uniformity} and Lemma \ref{lem: uniform ram properties} that, for any $i=1,\dots,n$, $l(i)$ does not depend on the disc $D$ that is attached to $\xi$. Then, we may write
$$
 \fna_{|D}=\bigoplus_{i=1}^n\bigoplus_{j=1}^{l(i)}(f_{i,j})_*(\cO_{D_{i,j}},d_{D_{i,j}}).
$$
By Theorem \ref{thm: properties p-de controlling} and Remark \ref{rmk: restr disc} it follows that
\begin{equation}\label{eq: ast radii precise}
\cM\cR_{\Gamma}(x,\fna)=\cM\cR(x,\fna_{|D})=\bigast_{i=}^n\bigast_{j=1}^{l(i)}\cM\cR(x,(f_{i,j})_*(\cO_{D_{i,j}},d_{D_{i,j}})).
\end{equation}
Since $f$ is radial, each of the morphism $f_{i,j}$ is radial and for a fixed $i$, the profile of $f_{i,j}$ does not depend on $j$. Consequently, for a fixed $i$, the functions $N_{f_{i,j},x}$ do not depend on $j$ and neither on $x$ nor on the disc $D$ in $C_{\xi}$ attached to $\xi$ (see Corollary \ref{cor: N properties} $(2)$). Finally, by Lemma \ref{lem: push constant} we may conclude that the right-hand side of \eqref{eq: ast radii precise} does not depend on $x\in C_\xi(k)$ so that $\Gamma_X$ is controlling for $\fna$. 
\par
In the other direction, suppose that $\Gamma_X$ is controlling for $\fna$. We note that in order to prove that $f$ is radial it is enough to prove that for each $i=1,\dots,n$, $f_{|C_{\eta_i}}:C_{\eta_i}\to C_\xi$ is radial (with respect to the skeleton $(\{\eta_i\},\{\xi\})$). The fact that $\Gamma_X$ is controlling for $\fna$, hence that $\{\xi\}$ is controlling for $\fna_{|C_\xi}$ together with \eqref{eq: sum} implies that $\{\xi\}$ controls each of the $p$-adic differential equations $f_*(\cO_{C_{\eta_i}}, d_{C_{\eta_i}})$ by Theorem \ref{thm: properties p-de controlling}. This means that, without loss of generality, we may assume that $f:Y\to X$ is a finite \'etale morphism of affinoid domains with good reduction and with maximal (type 2) points $\eta$ and $\xi$, respectively, and our goal is to prove that it is radial with respect to the canonical skeleton $(\{\eta\},\{\xi\})$, assuming that $\{\xi\}$ controls $\fna$. 

By Lemma \ref{lem: push res ram} $f$ is residually uniformly ramified and by  Theorem \ref{thm: factorization} there exists a strictly  $k$-affinoid curve $Z$ with good reduction and Shilov point $\omega$ together with finite \'etale morphisms $f_i:Y\to Z$ and $f_s:Z\to X$ such that $f_i$ is residually radicial, $f_s$ is residually \'etale, and $f=f_s\circ f_i$. Since $f_s$ is radial with respect to $(\{\eta\},\{\xi_s\})$ (Corollary \ref{cor: res etale radial}) then $f$ will be radial if and only if $f_i$ is radial (Remark \ref{rmk: 2 out 3}). Moreover, since $\fna=(f_{s})_*\big((f_{i})_*(\cO_Y,d_Y)\big)$ Corollary \ref{cor: multi constant discs} implies that $\{\xi\}$ is controlling for $\fna$ (w.r.t. $\{\xi\}$) if and only if $\{\omega\}$ is controlling for $(f_{i})_*(\cO_Y,d_Y)$ (w.r.t. $\{\omega\}$). In other words, without loss of generality we may assume that our morphism $f:Y\to X$ is in addition also residually radicial at $\eta$. 

Now, for any disc $B$ in $Y$ attached to $\eta$, the restriction $f_{|B}:B\to D:=f(B)$ is a finite \'etale morphism of open discs, and $f^{-1}(D)=B$ (Lemma \ref{lem: uniform ram properties}). Hence, $\fna_{|D}=(f_{|B})_*(\cO_B,d_B)$ and for any $x\in D$
$$
\cM\cR_{\{\xi\}}(x,\fna)=\cM\cR(x,\fna_{|D})=\cM\cR(x,(f_{|B})_*(\cO_B,d_B)).
$$
Since the left hand side of the previous equation is constant on $C_\xi$ it follows from Lemma \ref{lem: push constant} that the functions $N_{f_{|D},x}$ do not depend on $x$, hence $f_{|D}$ is radial. For the same reason the functions $N_{f_{D},x}$ coincide for all discs $D$ and $x\in D(k)$. Then Corollary \ref{cor: N properties} $(2)$ implies that for all $D$ the morphisms $f_{|D}$ have the same profile. The morphism $f$ is then radial.
\end{proof}
\begin{rmk}
The previous theorem together with \ref{thm: properties p-de controlling} {\em (1)} implies the existence of  radializing skeleta for finite \'etale morphisms of quasi-smooth strictly $k$-analytic curves, \cite[Theorem 3.4.11.]{TeMU}. One can also allow classical ramification. Given a finite $f:Y\to X$ morphism of quasi-smooth $k$-analytic curves, one restricts $f$ to a finite \'etale morphism $g: Y - f^{-1}(B) \to X - B$, where $B \subset X(k)$ denotes  the branching locus of $f$. Once a radializing skeleton  for $g$  is obtained, a skeleton   for $f$ is found by adding edges  to reach the points of $ f^{-1}(B)$ and $B$, respectively.
\end{rmk}


\begin{thebibliography}{10}
%
 %
%
%


%
%
%
%
 

\bibitem{Cont}
Francesco Baldassarri.
\newblock Continuity of the radius of convergence of  differential equations on $p$-adic analytic curves
\newblock {\em Invent. Math.}
182(3): 513-584, 2010.


\bibitem{Ram}
Francesco Baldassarri.
\newblock Radius of convergence of $p$-adic connections
and  the $p$-adic Rolle theorem
\newblock  {\em Milan J. Math.}
81: 397-419, 2013.


%
%

\bibitem{Berkovich}
Vladimir~G. Berkovich.
\newblock {\em Spectral theory and analytic geometry over non-{A}rchimedean
  fields}, volume~33 of {\em Mathematical Surveys and Monographs}.
\newblock American Mathematical Society, Providence, RI, 1990.

\bibitem{BerkovichEtale}
Vladimir~G. Berkovich.
\newblock \'{E}tale cohomology for non-{A}rchimedean analytic spaces.
\newblock {\em Publ. Math. IHES}, 78:5--161, 1993.

 \bibitem{BerkovichCycles}
 Vladimir~G. Berkovich.
 \newblock Vanishing cycles for formal schemes.
 \newblock {\em Inventiones Mathematicae}, 115(3):539--571, 1994.
%
%
%

%

\bibitem{BojFact}
Velibor Bojkovi\'c.
\newblock Canonical factorization of morphisms of Berkovich curves,
\newblock{\em In preparation}



\bibitem{BoPoPush}
Velibor Bojkovi\'c, J\'er\^ome Poineau.
\newblock Pushforwards of {$p$}-adic differential equations,
\newblock {\em To appear in Amer. J. of Math.}
\newblock arXiv:1703.04188v3 

\bibitem{BGR}
Siegfried Bosch, Ulrich G\"untzer, Reinhold Remmert.
\newblock Non-Archimedean Analysis. A Systematic Approach to Rigid Analytic Geometry,
\newblock {\em Grundlehren der Mathematischen Wissenschaften.}
\newblock{Springer-Verlag, Berlin}
vol 261, 1986.


\bibitem{CTT}
 Adina Cohen, Michael Temkin, Dmitri Trushin.
 \newblock Morphisms of Berkovich curves and the different function. 
 \newblock  Adv. Math. 303 (2016), p. 800--858. 

\bibitem{Col85}
Robert Coleman
\newblock  Torsion points on curves and $p$-adic Abelian integrals,
\newblock {\em Annals of Math. }
Vol. 121(1): 111--168, 1995 

\bibitem{Col03} 
 Robert F. Coleman.
 \newblock Stable maps of curves
 \newblock {\em Documenta Math. Extra Volume Kato}
 217--225, 2003.

\bibitem{ducros}
Antoine Ducros.
\newblock{\em  La structure des courbes analytiques},
\newblock manuscript in preparation, available online at {\em http://webusers.imj-prg.fr/~antoine.ducros/livre.html}.
\newblock Accessed November 2018.

\bibitem{sga1} 
Alexandre Grothendieck
 \newblock S\'eminaire de G\'eometrie Alg\'ebrique I. Rev\^etements \'Etales et Groupe Fondamental
 \newblock {\em Lecture Notes in Math. 224}
  \newblock {\em Springer}
 1971.
 
%
%
\bibitem{Ke2}                                                                                                         
 Kiran S. Kedlaya.
\newblock {\em $p$-adic Differential Equations},
\newblock volume 125 of {\em Cambridge Studies in Advanced Mathematics}, Cambridge Univ. Press, Cambridge, 2010. 

\bibitem{Ke5}
Kiran S. Kedlaya.
 \newblock Local and global structure of connections on nonarchimedean curves,
 \newblock {\em Compos. Math.} 151(6): 1096--1156, 2015.

\bibitem{lazard}
Michel Lazard 
\newblock Les z\'eros d'une fonction analytique d'une variable sur un corps valu\'e complet.
\newblock {\em Publ. Math. IHES}, 14:47--75, 1962. 

\bibitem{Liu}
Qing Liu.
\newblock  Algebraic geometry and arithmetic curves
\newblock {\em Oxford Graduate Texts in Mathematics 6}
Oxford University Press, 2002

%
%
%
 
\bibitem{PoPu} J\'er\^{o}me Poineau and 
Andrea Pulita.
\newblock Finiteness of the convergence Newton polygon of a $p$-adic differential equation II: continuity and finiteness on  Berkovich curves.
\newblock {\em Acta Math.} 214: 357--393, 2015.

\bibitem{Pu1}
Andrea Pulita.
\newblock The convergence Newton polygon of a $p$-adic differential
              equation I: Affinoid domains of the Berkovich affine
              line,
\newblock {\em Acta Math.} 214: 307--355, 2015.


 \bibitem{stacks}
 The Stacks project. 
\newblock {https://stacks.math.columbia.edu/browse}


 \bibitem{TeMU}
 Michael Temkin.
 \newblock  Metric uniformization of morphisms of Berkovich curves.
\newblock {\em Adv. in Math.} 317: 438--472, 2017.

%

\end{thebibliography}
\end{document}